\documentclass[final]{siamltex}
\usepackage{amsmath}
\usepackage{amsfonts}
\usepackage{graphicx}
\newcommand*{\K}{{\mathsf{K}}}
\newcommand*{\I}{{\mathsf{I}}}

\newcommand{\mE}{{\mathcal{E}}}
\newcommand{\mU}{{\mathcal{U}}}

\newcommand{\nn}{{\mathrm{n}}}
\newcommand{\magnus}{{\mathrm{mag}}}
\newcommand{\neu}{{\mathrm{neu}}}
\newcommand{\tr}{{\mathrm{T}}}

\newcommand{\bE}{{\mathbb{E}}}
\newcommand{\eval}{{\mathrm{eval}}}
\newcommand{\quadr}{{\mathrm{quad}}}
\newcommand{\dpp}{{\ell}}

\title{Efficient strong integrators for linear stochastic systems}
\author{Gabriel Lord\thanks{Maxwell Institute for Mathematical Sciences
and School of Mathematical and Computer Sciences,
Heriot-Watt University, Edinburgh EH14 4AS, UK 
(\texttt{G.J.Lord@hw.ac.uk}, 
\texttt{S.J.Malham@ma.hw.ac.uk}, 
\texttt{A.Wiese@hw.ac.uk}). (21/8/2007)} \and
Simon J.A. Malham$^\ast$\thanks{SJAM would like to dedicate this paper 
to the memory of Nairo Aparicio, a friend and collaborator 
who passed away on 20th June 2005.}
\and Anke Wiese$^\ast$}

\begin{document}

\maketitle
\label{firstpage}

\begin{abstract}
We present numerical schemes for the strong solution of linear 
stochastic differential equations driven by an arbitrary
number of Wiener processes. These schemes are based on the 
Neumann (stochastic Taylor) and Magnus expansions. Firstly, we consider
the case when the governing linear diffusion vector fields commute
with each other, but not with the linear drift vector field. 
We prove that numerical methods based on the Magnus expansion are  
more accurate in the mean-square sense than corresponding 
stochastic Taylor integration schemes. 
Secondly, we derive the maximal rate of convergence for 
arbitrary multi-dimensional stochastic integrals
approximated by their conditional expectations.
Consequently, for general nonlinear 
stochastic differential equations with non-commuting  
vector fields, we deduce explicit formulae for the relation between 
error and computational costs for  methods of arbitrary
order. Thirdly, we consider the consequences in two 
numerical studies, one of which is an application arising in 
stochastic linear-quadratic optimal control.
\end{abstract}



\pagestyle{myheadings}
\thispagestyle{plain}
\markboth{Lord, Malham and Wiese}{Efficient stochastic integrators}

\section{Introduction}
We are interested in designing efficient numerical schemes
for the strong approximation 
of linear Stratonovich stochastic differential
equations of the form
\begin{equation}\label{sde}
y_t=y_0+\sum_{i=0}^d\int_0^t a_i(\tau)\,y_\tau\,\mathrm{d}W^i_\tau\,,
\end{equation}
where $y\in\mathbb R^p$, $W^0\equiv t$, $(W^1,\ldots,W^d)$ 
is a $d$-dimensional Wiener process and $a_0(t)$ and $a_i(t)$ 
are given $p\times p$ coefficient matrices. We call `$a_0(t)\,y$'
the linear \emph{drift} vector field and `$a_i(t)\,y$' for 
$i=1,\ldots,d$ the linear \emph{diffusion} vector fields.
We can express the stochastic differential equation~\eqref{sde} 
more succinctly in the form
\begin{equation}\label{sdeabs}
y=y_0+\K\circ y\,,
\end{equation}
where $\K\equiv \K_0+\K_1+\cdots+\K_d$ and 
$(\K_i\circ y)_t\equiv\int_0^t a_i(\tau)\,y_\tau\,\mathrm{d}W^i_\tau$.
The solution of the integral equation for $y$
is known as the \emph{Neumann series}, \emph{Peano--Baker series},
\emph{Feynman--Dyson path ordered exponential}
or \emph{Chen-Fleiss series}  
\begin{equation*}
y_t=(\I-\K)^{-1}\circ y_0\equiv(\I+\K+\K^2+\K^3+\cdots)\circ y_0\,.
\end{equation*}
The flow-map or fundamental solution
matrix $S_t$ maps the initial data $y_0$ to the 
solution $y_t=S_t\,y_0$ at time $t>0$. It satisfies an
analogous matrix valued stochastic differential equation
to~\eqref{sdeabs} with the $p\times p$ 
identity matrix as initial data.
The logarithm of the Neumann expansion for the flow-map is the
\emph{Magnus expansion}. We can thus write the solution to 
the stochastic differential equation~\eqref{sde} in the form
\begin{equation*}
y_t=(\exp\sigma_t)\,y_0\,,
\end{equation*}
where
\begin{equation}\label{Magnus}
\sigma_t=\ln\bigl((\I-\K)^{-1}\circ I\bigr)
\equiv\K\circ I+\K^2\circ I
-\tfrac12(\K\circ I)^2+\cdots\,.
\end{equation}
See Magnus~\cite{Ma}, Kunita~\cite{Ku}, Azencott~\cite{Az}, 
Ben Arous~\cite{BA}, Strichartz~\cite{Str}, Castell~\cite{C}, 
Burrage~\cite{B}, Burrage and Burrage~\cite{BB} and Baudoin~\cite{Ba}
for the derivation and convergence of the original 
and also stochastic Magnus expansion;
Iserles, Munthe--Kaas, N\o rsett and Zanna~\cite{IMNZ} for a 
deterministic review; Lyons~\cite{L} and Sipil\"ainen~\cite{Si} 
for extensions to rough signals; 
Lyons and Victoir~\cite{LV} for a recent application to probabilistic
methods for solving partial differential equations; and 
Sussmann~\cite{Su} for a related product expansion.

In the case when the coefficient matrices
$a_i(t)=a_i$, $i=0,\ldots,d$ are constant and 
non-commutative, the solution to the
linear problem~\eqref{sde} is non-trivial 
and given by the Neumann series 
or stochastic Taylor expansion
(see Kloeden and Platen~\cite{KP})
\begin{equation}\label{Neumann}
y^{\text{neu}}_t=\sum_{\ell=0}^\infty\,\sum_{\alpha\in\mathbb P_\ell}
J_{\alpha_\ell\cdots\alpha_1}(t)\,a_{\alpha_1}\cdots a_{\alpha_\ell}\,y_0\,,
\end{equation}
where
\begin{equation*}  
J_{\alpha_\ell\cdots\alpha_1}(t)\equiv 
\int_{0}^{t}\int_{0}^{\xi_1}\cdots\int_{0}^{\xi_{\ell-1}}\,
\mathrm{d}W^{\alpha_\ell}_{\xi_\ell}\cdots\mathrm{d}W^{\alpha_2}_{\xi_2}
\,\mathrm{d}W^{\alpha_1}_{\xi_1}\,.
\end{equation*}
Here $\mathbb P_\ell$ is the set of all combinations
of multi-indices $\alpha=\{\alpha_1,\ldots,\alpha_\ell\}$ of length $\ell$
with $\alpha_k\in\{0,1,\ldots,d\}$ for $k=1,\ldots,\ell$. 
There are some special non-commutative cases when we can
write down an explicit analytical solution.
For example the stochastic differential equation
$\mathrm{d}y_t=a_1y_t\,\mathrm{d}W^1_t+y_ta_2\,\mathrm{d}W^2_t$
with the identity matrix as initial data 
has the explicit analytical solution
$y_t=\exp(a_1W_t^1)\cdot\exp(a_2W_t^2)$.
However in general we cannot express the Neumann solution
series~\eqref{Neumann} in such a closed form.

Classical numerical schemes such as the Euler-Maruyama 
and Milstein methods correspond to truncating the 
stochastic Taylor expansion to generate global 
strong order $1/2$ and order $1$ schemes, respectively.
Stochastic Runge--Kutta numerical methods
have also been derived---see Kloeden and Platen~\cite{KP} and 
Talay~\cite{T}. At the linear level, the Neumann, 
stochastic Taylor and Runge--Kutta type methods 
are equivalent. In the stochastic context, Magnus
integrators have been considered by Castell and Gaines~\cite{CG},
Burrage~\cite{B}, Burrage and Burrage~\cite{BB} and Misawa~\cite{Mi}.

We present numerical schemes based on truncated Neumann
and Magnus expansions. 
Higher order multiple Stratonovich integrals are 
approximated across each time-step by their
expectations conditioned on the increments 
of the Wiener processes on 
suitable subdivisions (see Newton~\cite{N} 
and Gaines and Lyons~\cite{GL97}).
What is new in this paper is that we: 
\begin{enumerate}
\item Prove the strong convergence of the 
truncated stochastic Magnus expansion for small stepsize;
\item Derive uniformly accurate higher order stochastic integrators 
based on the Magnus expansion in the case of commuting
linear diffusion vector fields;
\item Prove the maximal rate of convergence for 
arbitrary multi-dimensional stochastic integrals
approximated by their conditional expectations;
\item Derive explicit formulae for the relation between 
error and computational costs for methods of arbitrary
order in the case of general nonlinear, non-commuting  
governing vector fields.
\end{enumerate}
Our results can be extended to nonlinear stochastic 
differential equations with analogous conditions on the 
governing nonlinear vector fields, where the
exponential Lie series (replacing the Magnus expansion)
can be evaluated using the Castell--Gaines approach.

In the first half of this paper, sections 2--5, we focus on
proving the convergence of the truncated Magnus
expansion and establishing Magnus integrators that
are more accurate than Neumann (stochastic Taylor)
schemes of the same order.
The numerical schemes we present belong to
the important class of \textit{asymptotically efficient} 
schemes introduced by Newton~\cite{N}. Such schemes
have the optimal minimum leading error coefficient among
all schemes that depend on increments of the underlying
Wiener process only.
Castell and Gaines~\cite{CG,CG2} prove that the
order $1/2$ Magnus integrator driven by a $d$-dimensional
Wiener process and a modified order $1$ Magnus integrator
driven by a $1$-dimensional Wiener process are 
asymptotically efficient. We extend this result of
Castell and Gaines to an arbitrary
number of driving Wiener processes. 
We prove that if we assume the linear diffusion
vector fields commute, then an analogously modified
order $1$ Magnus integrator and 
a new order~$3/2$ Magnus integrator are
globally more accurate than their corresponding
Neumann integrators.

There are several potential sources of cost contributing to the 
overall computational effort of a stochastic numerical
integration scheme. The main ones are the efforts associated with:
\begin{itemize}
\item Evaluation: computing (and combining) the individual terms and
special functions such as the matrix exponential;
\item Quadrature: the accurate representation of multiple Stratonovich integrals.
\end{itemize}
There are usually fewer terms in the Magnus expansion 
compared to the Neumann expansion to the same order,
but there is the additional
computational expense of computing the matrix exponential.
When the cost of computing the 
matrix exponential is not significant, 
due to their superior accuracy
we expect Magnus integrators to be preferable
to classical stochastic numerical integrators. 
This will be the case for systems
that are small (see Moler and Van Loan~\cite{MV} and 
Iserles and Zanna~\cite{IZ}) or for large systems 
when we only have to compute the exponential 
of a large sparse matrix times given vector data 
for which we can use Krylov subspace
methods (see Moler and Van Loan~\cite{MV} 
and Sidje~\cite{Sidje}). 
Magnus integrators are also preferable 
when using higher order integrators (applied
to non-sparse systems of any size) when
high accuracies are required. This is because
in this scenario, quadrature computational cost 
dominates integrator effort. 

In the second half of this paper, sections 6--8, we focus on the
quadrature cost associated with approximating multiple
Stratonovich integrals to a degree
of accuracy commensurate with the order of
the numerical method implemented. 
Our conclusions apply generally to the case of
nonlinear, non-commuting governing vector fields.
The governing set of vector fields and driving path process 
$(W^1,\ldots,W^d)$ generate
the unique solution process $y\in\mathbb R^p$ 
to the stochastic differential equation~\eqref{sde}.
For a scalar driving Wiener process $W$
the It\^o map $W\mapsto y$ is
continuous in the topology of uniform convergence.
For a $d$-dimensional driving processes with $d\geq2$
the Universal Limit Theorem implies that
the It\^o map $(W^1,\ldots,W^d)\mapsto y$ is continuous in the 
$p$-variation topology, in particular for 
$2\leq p<3$ (see Lyons~\cite{L}, Lyons and Qian~\cite{LQ} and 
Malliavin~\cite{Malliavin}). Since Wiener paths with $d\geq2$ 
have finite $p$-variation for $p>2$,
approximations to $y$ constructed using 
successively refined approximations to the driving path
will only converge to the correct
solution $y$ if we include information about
the L\'evy chordal areas of the driving path
(the $L^2$-norm of the $2$-variation of a Wiener process 
is finite though). Hence if we want to implement 
a scheme using adaptive stepsize 
we should consider order~$1$ or higher
pathwise stochastic numerical methods 
(see Gaines and Lyons~\cite{GL97}).
 
However simulating multiple Stratonovich integrals
accurately is costly! For classical accounts
of this limitation on applying higher order
pathwise stochastic numerical schemes
see Kloeden and Platen~\cite[p.~367]{KP}, 
Milstein~\cite[p.~92]{Mil} and Schurz~\cite[p.~58]{Sc} and 
for more recent results see
Gaines and Lyons~\cite{GL94,GL97}, Wiktorsson~\cite{W},
Cruzeiro, Malliavin and Thalmaier~\cite{CMT},
Stump and Hill~\cite{SH} and Giles~\cite{Gi1,Gi2}.

Taking a leaf from Gaines and Lyons~\cite{GL97} we 
consider whether it is computationally cheaper to collect
a set of sample data over a given time interval and then 
evaluate the solution (conditioned on that sample data),
than it is to evaluate the solution frequently, say
at every sample time.
The resounding result here is that of
Clark and Cameron~\cite{CC} who 
prove that when the multiple Stratonovich integral 
$J_{12}$ is approximated by its expectation
conditioned on intervening sample points,
the maximal rate of $L^2$-convergence 
is of order $h/Q^{1/2}$ where
$h$ is the integration steplength and $Q$ is
the sampling rate. 
We extend this result to multiple
Stratonovich integrals $J_{\alpha_1,\ldots,\alpha_\ell}$
of arbitrary order approximated by their expectation conditioned on 
intervening information sampled at the rate $Q$.
Indeed we prove that the maximal rate of convergence
is $h^{\ell/2}/Q^{1/2}$ when $\alpha_1,\ldots,\alpha_\ell$
are non-zero indices (and an improved rate of convergence
if some of them are zero). In practice the key information is 
how the accuracy achieved scales with the effort
required to produce it on the global interval of
integration say $[0,T]$ where $T=Nh$. We derive
an explicit formula for the relation between
the global error and the computational
effort required to achieve it for a 
multiple Stratonovich integral of arbitrary order
when the indices $\alpha_1,\ldots,\alpha_\ell$ are distinct.
This allows us to infer the effectiveness of
strong methods of arbitrary order for systems
with non-commuting vector fields. For a given
computational effort which method delivers the
best accuracy? The answer not only relies on
methods that are more accurate at a given order.
It also is influenced by three regimes for the 
stepsize that are distinguished as follows. 
In the first large stepsize regime the evaluation 
effort is greater than the quadrature effort; 
higher order methods produce superior performance
for given effort. Quadrature effort exceeds
evaluation effort in the second smaller stepsize
regime. We show that in this regime when $d=2$,
or when $d\geq3$ and the order of the method $M\leq3/2$,
then the global error scales with the computational
effort with an exponent of $-1/2$. Here more
accurate higher order methods still produce 
superior performance for given effort; but not
at an increasing rate as the stepsize is decreased.
However when $d\geq3$ for strong methods with 
$M\geq2$ the global error verses computational 
effort exponent is worse than $-1/2$ and this 
distinguishes the third very small stepsize regime.
The greater exponent means that eventually lower
order methods will deliver greater accuracy for
a given effort.

We have chosen to approximate higher order integrals 
over a given time step by their
expectations conditioned on the increments 
of the Wiener processes on suitable subdivisions. 
This is important for adaptive time-step schemes
(Gaines and Lyons~\cite{GL97}) and filtering problems
where the driving processes (say $W^1$ and $W^2$) are
observed signals. However it should be noted that
Wiktorsson~\cite{W} has provided a practical method
for efficiently sampling the set of multiple Stratonovich
multiple integrals $\{J_{ij}\colon i,j=1,\ldots,d\}$ 
across a given time-step
associated with a $d$-dimensional driving process 
(see Gilsing and Shardlow~\cite{GS} for a practical
implementation). 
Wiktorsson simulates the tail distribution in 
a truncated Karhunen--Loeve 
Fourier series approximation of these integrals
which produces a convergence rate of order
$hd^{3/2}/Q$ where $Q$ is analogously 
the number of required
independent normally distributed samples. 

Other potential sources of computational effort
might be path generation and memory access. 
Path generation effort depends on the application context. 
This cost is at worst proportional to the quadrature effort where 
we could subsume it. Memory access efforts 
depend on the processing and access memory environment. 
To reveal higher order methods (which typically require
more path information) in the best light possible, 
we have ignored this effect.

Our paper is outlined as follows. We start in Section 2
by proving that the exponential of every truncation of the
Magnus series converges to the solution of our 
linear stochastic differential equation~\eqref{sde}.
In Section~3 we define the strong error measures we use 
and how to compute them. Using these, we 
explicitly compare the local and then global errors 
for the Magnus and Neumann integrators in Section~4
and thus establish our stated results for 
uniformly accurate Magnus integrators. 
In Section~5 we show that when the linear diffusion vector fields
do not commute we cannot expect the corresponding 
order~$1$ Magnus integrator to in general be 
globally more accurate than the order~$1$ Neumann integrator.
We then turn our attention in Section~6 to
the method of approximating multiple Stratonovich integrals by
their conditional expectations. 
We prove the maximal rate of convergence for
an arbitrary multiple Stratonovich integral
in Section~6. We then use this result in Section~7 to show
how the global error scales with the computational effort
for numerical schemes of arbitrary order.
The shuffle algebra of multiple Stratonovich integrals
generated by integration by parts allows for different
representions and therefore bases for the solution
of a stochastic differential equation. Some choices
of basis representation are more efficiently 
approximated than others and we investigate in Section~8 the
impact of this choice.
In Section~9 we present numerical experiments that reflect
our theoretical results. To illustrate the superior
accuracy of the uniformly accurate Magnus methods 
we apply them to a stochastic 
Riccati differential system that can be reformulated
as a linear system which has
commuting diffusion vector fields. 
Since for the linear system expensive matrix-matrix 
multiplications can be achieved independent of the path,
the Neumann method performs better than 
an explicit Runge--Kutta type method applied 
directly to the nonlinear Riccati system.
We also numerically solve an explicit linear
system with governing linear vector fields that
do not commute for two and also three driving
Wiener processes---Magnus integrators also
exhibit superior accuracy in practice 
in these cases also. 
Lastly in Section~10, we outline how to extend
our results to nonlinear stochastic differential
equations and propose further extensions 
and applications.

\section{Strong convergence of truncated Magnus series}
\label{Convergence}
We consider here the case when the 
stochastic differential equation~\eqref{sde} 
is driven by $d$ Wiener processes with
constant coefficient matrices $a_i(t)=a_i$, $i=0,1,\ldots,d$. 
The Neumann expansion has the form shown in~\eqref{Neumann}.
We construct the Magnus expansion by taking the 
logarithm of this Neumann series as in~\eqref{Magnus}.
In Appendix~\ref{theapp} we explicitly give the 
Neumann and Magnus expansions
up to terms with $L^2$-norm of order $3/2$.
Let $\hat\sigma_t$ denote the truncated Magnus series
\begin{equation}\label{eq:orderp}
\hat\sigma_t=\sum_{\alpha \in \mathbb Q_m} J_\alpha(t)\, c_\alpha\,,
\end{equation}
where $\mathbb Q_m$ denotes the finite set of multi-indices $\alpha$
for which $\|J_\alpha\|_{L^2}$ is of order up to and including $t^m$. 
Note that here $m$ is a half-integer index, $m=1/2,1,3/2,\ldots$. 
The terms $c_\alpha$ are linear combinations 
of finitely many (more precisely exactly length $\alpha$)
products of the $a_i$, $i=0,1,\ldots,d$. Let $|\mathbb Q_m|$ denote the
cardinality of $\mathbb Q_m$. 

\begin{theorem}[Convergence]\label{th:conv}
For any $t\leq1$, the exponential of the
truncated Magnus series, $\exp\hat\sigma_t$, 
is square-integrable. Further, if $y_t$ is the solution 
of the stochastic differential equation~\eqref{sde},
there exists a constant $C(m)$ such that
\begin{equation}\label{conv}
\bigl\|y_t-\exp\hat\sigma_t\cdot y_0\bigr\|_{L^2}
\leq C(m)\,t^{m+1/2}\,.
\end{equation}
\end{theorem}

\begin{proof}
First we show that $\exp\hat\sigma_t\in L^2$. 
Using the expression \eqref{eq:orderp} for $\hat\sigma_t$, 
we see that for any number $k$, $(\hat\sigma_t)^k$ 
is a sum of $|\mathbb Q_m|^k$ terms,
each of which is a $k$-multiple product of terms $J_\alpha c_\alpha$.
It follows that
\begin{equation} \label{eq:product}
\bigl\|(\hat\sigma_t)^k\bigr\|_{L^2}\leq 
\Bigl(\,\max_{\alpha\in\mathbb Q_m}\|c_\alpha\|_{\mathrm{op}}\Bigr)^k \,\cdot
\sum_{\stackrel{\alpha_i \in\mathbb Q_m}{i=1,\ldots,k}}
\|J_{\alpha(1)}J_{\alpha(2)}\cdots J_{\alpha(k)}\|_{L^{2}}\,.
\end{equation}
Note that the maximum of the operator norm $\|\,\cdot\,\|_{\mathrm{op}}$ 
of the coefficient matrices is taken over a finite set.
Repeated application of the product rule reveals that
the product $J_{\alpha(i)} J_{\alpha(j)}$, where $\alpha(i)$ and $\alpha(j)$
are multi-indices of length $\ell(\alpha(i))$ and $\ell(\alpha(j))$, 
is a linear combination of $2^{\ell(\alpha(i))+\ell(\alpha(j))-1}$ 
multiple Stratonovich integrals. Since $\ell(\alpha(i))\leq 2m$ 
for $i=1,\ldots,k$, each term `$J_{\alpha(1)}J_{\alpha(2)}\cdots J_{\alpha(k)}$' 
in \eqref{eq:product} is thus the sum of at most 
$2^{2mk-1}$ Stratonovich integrals $J_\beta$. We also note
that $k\leq\ell(\beta)\leq 2mk$.

From equation (5.2.34) in Kloeden and Platen~\cite{KP}, 
every multiple Stratonovich integral $J_{\beta}$ 
can be expressed as a finite sum of at most 
$2^{\ell(\beta)-1}$ multiple It{\^o} integrals
$I_{\gamma}$ with $\ell(\gamma)\leq\ell(\beta)$.
Further, from Remark 5.2.8 in Kloeden and Platen~\cite{KP},
$\ell(\gamma)+\nn(\gamma)\geq\ell(\beta)+\nn(\beta)$, where
$\nn(\beta)$ and $\nn(\gamma)$ denote
the number of zeros in $\beta$ and $\gamma$, respectively.
From Lemma 5.7.3 in Kloeden and Platen~\cite{KP},
\begin{equation*}
\|I_\gamma\|_{L^2}\leq 2^{\ell(\gamma)-\nn(\gamma)}\,
t^{(\ell(\gamma)+\nn(\gamma))/2}\,.
\end{equation*}
Noting that $\ell(\gamma)\leq\ell(\beta)\leq 2mk$
and $\ell(\gamma)+\nn(\gamma)\geq k$, 
it follows that for $t\leq1$, we have
$\|J_{\beta}\|_{L^{2}}\leq2^{4mk-1}\,t^{k/2}$.
Since the right hand side of equation \eqref{eq:product}
consists of $|\mathbb Q_m|^k\,2^{2mk-1}$ Stratonovich integrals $J_\beta$,
we conclude that,
\begin{equation*}
\Bigl\|\bigl(\hat\sigma_t\bigr)^k\Bigr\|_{L^2}
\leq\Bigl(\,\max_{\alpha\in\mathbb Q_m}\|c_\alpha\|_{\mathrm{op}}
\cdot|\mathbb Q_m|\cdot2^{6m}\cdot t^{1/2}\Bigr)^k\,.
\end{equation*}
Hence $\exp\,\hat\sigma_t$ is square-integrable.

Second we prove~\eqref{conv}. Let $\hat y_t$ denote 
Neumann series solution~\eqref{Neumann} truncated 
to included terms of order up to and including $t^m$.
We have
\begin{equation}\label{convinequ}
\bigl\|y_t-\exp\hat\sigma_t\cdot y_0\bigr\|_{L^2}\leq
\bigl\|y_t-\hat y_t\bigr\|_{L^2}
+\bigl\|\hat y_t-\exp\hat\sigma_t\cdot y_0\bigr\|_{L^2}\,.
\end{equation}
We know $y_t\in L^2$ (see Gihman and Skorohod~\cite{GS} 
or Arnold~\cite{A}). Furthermore, for any order $m$,
$\hat y_t$ corresponds to the truncated Taylor expansion 
involving terms of order up to and including $t^m$.
Hence $\hat y_t$ is a strong approximation to $y_t$ to that order
with the remainder consisting of $\mathcal O(t^{m+1/2})$ terms
(see Proposition 5.9.1 in Kloeden and Platen~\cite{KP}).
It follows from the definition of the Magnus series as
the logarithm of the flow-map Neumann series, 
that the terms of order up to and including $t^m$
in $\exp\hat\sigma_t\cdot y_0$ correspond with $\hat y_t$; 
the error consists of 
$\mathcal O(t^{m+1/2})$ terms.\hfill \qquad\end{proof}

Convergence of approximations based on truncations 
of the stochastic Taylor expansion has been
studied in Kloeden and Platen~\cite{KP}, see 
Propositions 5.10.1, 5.10.2, and 10.6.3. Ben
Arous~\cite{BA} and Castell~\cite{C} prove the 
remainder  of the exponential of any truncation
of the Magnus series is bounded in probability 
as $t\rightarrow0$ (in the full nonlinear case).
Burrage~\cite{B} shows that the first terms up 
to and including order $3/2$ Magnus
expansion coincide with the terms in the Taylor 
expansion of the same order. Our result holds
for any order in $L^2$ for sufficiently small $t$. 
A more detailed analysis is needed to
establish results concerning the convergence radius. 
Similar arguments can be used to study the
non-constant coefficient case with suitable 
conditions on the coefficient matrices (see
Proposition 5.10.1 in Kloeden and Platen~\cite{KP} 
for the corresponding result for the Taylor
expansion).

Note that above and in subsequent sections, one may
equally consider a stochastic differential equation starting
at time $t_0>0$ with square-integrable $\mathcal F_{t_0}$-measurable
initial data $y_0$. Here $(\mathcal F_{t})_{t\geq0}$ denotes
the underlying filtration.

\section{Global and local error}\label{sec:globalerror}
Suppose $S_{t_n,t_{n+1}}$ and $\hat S_{t_n,t_{n+1}}$ are
the exact and approximate flow-maps across the
interval $[t_n,t_{n+1}]$, respectively;
both satisfying the usual flow-map semi-group property:
composition of flow-maps across successive intervals
generates the flow-map across the union of those intervals. 
We call the difference between the exact
and approximate flow-maps  
\begin{equation}\label{lapp} 
R_{t_n,t_{n+1}}\equiv S_{t_n,t_{n+1}}-\hat S_{t_n,t_{n+1}}
\end{equation}
the \emph{local flow remainder}. For an approximation
$\hat y_{t_{n+1}}$ across the interval $[t_n,t_{n+1}]$
the local remainder is thus $R_{t_n,t_{n+1}}y_{t_n}$. 
Our goal here is to see how the leading order 
terms in the local remainders accumulate,
contributing to the global error.

\begin{definition}[Strong global error]
We define the \emph{strong global error} associated with
an approximate solution $\hat y_T$ to the stochastic differential
equation~\eqref{sde} over the global interval of integration 
$[0,T]$ as $\mE\equiv\|y_T-\hat y_T\|_{L^2}$.
\end{definition}

The global error can be decomposed additively into
two components, the global truncation error due to 
truncation of higher order terms, and the global quadrature error
due to the approximation of multiple Stratonovich
integrals retained in the approximation. If
$[0,T]=\cup_{n=0}^{N-1}[t_n,t_{n+1}]$ where $t_n=nh$ then 
for small stepsize $h=T/N$ we have
\begin{align}
\mE=&\;\left\|\left(\prod_{n=N-1}^0S_{t_n,t_{n+1}}
 -\prod_{n=N-1}^0\hat S_{t_n,t_{n+1}}\right)y_0\right\|_{L^2}\notag\\
=&\;\left\|\left(\sum_{n=0}^{N-1}
\hat S_{t_{n+1},t_N} R_{t_n,t_{n+1}} 
\hat S_{t_0,t_n}\right)y_0\right\|_{L^2}
+\mathcal O\bigl(\max_n\|R_{t_n,t_{n+1}}\|_{L^2}^{3/2}\,h^{-3/2}\bigr)\,.
\label{errorest}
\end{align}

The local flow remainder has the following form 
in the case of constant coefficients $a_i$, $i=1,\ldots,d$,
(see for example the integrators in Appendix~\ref{theapp}):
\begin{equation*}
R_{t_n,t_{n+1}}=\sum_{\alpha} J_{\alpha}(t_n,t_{n+1})\,c_\alpha\,.
\end{equation*}
Here $\alpha$ is a multi-index and the  
terms $c_\alpha$ represent products or 
commutations of the constant matrices $a_i$.
The $J_\alpha$ represent Stratonovich integrals 
(or linear combinations, of the same order, of products of
integrals including permutations of $\alpha$).
The global error $\mE^2$ at leading order in the stepsize is thus
\begin{gather*}
y_0^\tr\sum_{\alpha,\beta}\biggl(\sum_{n}
\bE\bigl(J_{\alpha}(t_n,t_{n+1})\,J_{\beta}(t_n,t_{n+1})\bigr)
\bE\Bigl(\bigl(\hat S_{t_{n+1},t_N}c_\alpha \hat S_{t_0,t_n}\bigr)^\tr\!
\bigl(\hat S_{t_{n+1},t_N}c_\beta \hat S_{t_0,t_n}\bigr)\Bigr)\phantom{y_0}\\
+\!\!\sum_{n\neq m}\!\bE\bigl(J_{\alpha}(t_n,t_{n+1})\bigr)\,
\bE\bigl(J_{\beta}(t_{m},t_{m+1})\bigr)
\bE\Bigl(\bigl(\hat S_{t_{n+1},t_N}c_\alpha\hat S_{t_0,t_n}\bigr)^\tr\!
\bigl(\hat S_{t_{m+1},t_N}c_\beta\hat S_{t_0,t_{m}}\bigr)\Bigr)\biggr)y_0.
\end{gather*}
Hence in the global truncation error we distinguish between
the \emph{diagonal sum} consisting of the the first sum on the 
right-hand side above, and the \emph{off-diagonal sum} consisting
of the second sum above with $n\neq m$. 

Suppose we include in our integrator
all terms with local $L^2$-norm 
up to and including $\mathcal O(h^M)$.
The leading terms $J_\alpha c_\alpha$ in $R_{t_n,t_{n+1}}$ 
thus have $L^2$-norm $\mathcal O(h^{M+1/2})$. 
Those with zero expectation will contribute to the diagonal
sum, generating $\mathcal O(h^M)$ terms in the global error,
consistent with a global order $M$ integrator. However 
those with with non-zero 
expectation contribute to the off-diagonal double sum. They
will generate $\mathcal O(h^{M-1/2})$ terms in the global error.
We must thus either include them in the integrator, or more
cheaply, only include their expectations---the corresponding 
terms $\bigl(J_\alpha-\mathbb E(J_\alpha)\bigr) c_\alpha$ of order $h^{M+1/2}$ 
in $R_{t_n,t_{n+1}}$ will then have zero expectation and only
contribute through the diagonal sum---see 
for example Milstein~\cite[p.~12]{Mil}.
This also guarantees the next order term 
in the global error estimate~\eqref{errorest},
whose largest term has the upper bound
$\max_n\|R_{t_n,t_{n+1}}\|_{L^2}^{3/2}\,h^{-3/2}$, 
only involves higher order contributions to the
leading $\mathcal O(h^M)$ error.

Note that high order integrators may include multiple
Stratonovich integral terms. We approximate these multiple
integrals by their conditional expectations to the
local order of approximation $h^{M+1/2}$ of 
the numerical method. Hence
terms in the \emph{integrator} of the form $J_\alpha c_\alpha$
are in fact approximated by $\mathbb E(J_\alpha|\mathcal F_Q) c_\alpha$,
their expectation conditioned on 
intervening path information $\mathcal F_Q$
(see Section~\ref{sec:quadrature} for more details). 
This generates terms of the form 
$\bigl(J_\alpha-\mathbb E(J_\alpha|\mathcal F_Q)\bigr) c_\alpha$
in the local flow remainder, which have zero expectation
and hence contribute to the global error through the diagonal sum
generating $\mathcal O(h^M)$ terms.

\section{Uniformly accurate Magnus integrators}\label{globalcomparison}
Our goal is to identify a class Magnus integrators
that are more accurate than Neumann (stochastic Taylor)
integrators of the same order for any governing set
of linear vector fields (for the integrators of 
order $1$ and $3/2$ we assume the diffusion 
vector fields commute). We thus compare 
the local accuracy of the Neumann and Magnus integrators
through the leading terms of their remainders. 
We consider the case of constant coefficient matrices 
$a_i$, $i=0,1,\ldots,d$. 

The local flow remainder of a Neumann integrator $R^{\text{neu}}$ 
is simply given by the terms not included in the flow-map
Neumann approximation. Suppose $\hat\sigma$
is the truncated Magnus expansion 
and that $\rho$ is the corresponding remainder, 
i.e.\ $\sigma=\hat\sigma+\rho$.
Then the local flow remainder $R^{\text{mag}}$ 
associated with the Magnus approximation is
\begin{align}
R^{\text{mag}}=&\;\exp\sigma-\exp\hat\sigma\notag\\
=&\;\exp\bigl(\hat\sigma+\rho\bigr)-\exp\hat\sigma\notag\\
=&\;\rho+\tfrac12(\hat\sigma\rho+\rho\hat\sigma)
+\mathcal O(\hat\sigma^2 \rho)\,.\label{origmer}
\end{align}
Hence the local flow remainder of 
a Magnus integrator $R^{\text{mag}}$ 
is the truncated Magnus expansion remainder $\rho$, and 
higher order terms $\tfrac12(\hat\sigma\rho+\rho\hat\sigma)$
that can contribute to the global error at leading order through their
expectations. For the integrators considered in this section these
higher order terms do not contribute in this way,
however for the order~$1$ integrator we consider in the 
next section they do.

\begin{definition}[Uniformly accurate Magnus integrators]
When the linear diffusion vector fields commute so that 
$[a_i,a_j]=0$ for all $i,j\neq0$,
we define the order~$1$ and order $3/2$ 
\emph{uniformly accurate Magnus integrators} by 
\begin{equation*}
\hat\sigma^{(1)}_{t_n,t_{n+1}}=J_0a_0
+\sum_{i=1}^d\bigl(J_ia_i+\tfrac{h^2}{12}[a_i,[a_i,a_0]]\bigr)\,,
\end{equation*}
and
\begin{equation*}
\hat\sigma^{(3/2)}_{t_n,t_{n+1}}=J_0a_0
+\sum_{i=1}^d\bigl(J_ia_i+\tfrac12(J_{i0}-J_{0i})[a_0,a_i]
+\tfrac{h^2}{12}[a_i,[a_i,a_0]]\bigr)\,.
\end{equation*}
\end{definition}

By \emph{uniformly} we mean for any given set of governing
linear vector fields (or equivalently coefficient matrices $a_i$, 
$i=0,1,\ldots,d$) for which the diffusion vector fields commute, 
and for any initial data $y_0$.

\begin{theorem}[Global error comparison]\label{th:uls}
For any initial condition $y_0$ and 
sufficiently small fixed stepsize $h=t_{n+1}-t_n$,
the order~$1/2$ Magnus integrator is globally more
accurate in $L^2$ than the order~$1/2$ Neumann integrator.
If in addition we assume
the linear diffusion vector fields commute so that 
$[a_i,a_j]=0$ for all $i,j\neq 0$,
then the order~$1$ and $3/2$ 
\emph{uniformly accurate Magnus integrators} 
are globally more accurate in $L^2$ than the corresponding
Neumann integrators. In other words,
if $\mathcal E^{\mathrm{mag}}$ denotes the global error of 
the order~$1/2$ Magnus integrator or the 
uniformly accurate Magnus integrators of order~$1$ or
order $3/2$, respectively,
and $\mathcal E^{\mathrm{neu}}$ is the global error 
of the Neumann integrators of the 
corresponding order, then at each of those orders,
\begin{equation}\label{eq:globcom}
\mathcal E^{\mathrm{mag}}\leq\mathcal E^{\mathrm{neu}}\,.
\end{equation}
\end{theorem}

\begin{proof}
Let $R^\magnus$ and $R^\neu$ denote the 
local flow remainders corresponding to the 
Magnus and Neumann approximations across the 
interval $[t_n,t_{n+1}]$ with $t_n=nh$.
A direct calculation reveals that
\begin{equation*}
\mathbb E\bigl((R^\neu)^\tr R^\neu\bigr)=
\mathbb E\bigl((R^\magnus)^\tr R^\magnus\bigr)
+D_M+\mathcal O\bigl(h^{2M+1/2}\bigr)\,,
\end{equation*}
where if we set $\hat R\equiv R^\neu-R^\magnus$
then 
\begin{equation}\label{eq:dmexplicit}
D_M\equiv\mathbb E\bigl(\hat R^\tr R^\magnus\bigr)
+\mathbb E\bigl((R^\magnus)^\tr \hat R\bigr)
+\mathbb E\bigl(\hat R^\tr \hat R\bigr)\,.
\end{equation}
We now show explicitly in each of the three
cases of the theorem that at leading order: 
$R^\magnus$ and $\hat R$ are uncorrelated and 
hence $D_M$ is positive semi-definite.
This implies that the local remainder for the
Neumann expansion is larger than that of the
Magnus expansion.
Hereafter assume the indices $i,j,k,l\in\{1,\ldots,d\}$.

For the \emph{order~$1/2$ integrators} we have to leading order
(see the form of the expansions in Appendix~A)
\begin{equation*}
R^\magnus=\sum_{i<j}\tfrac12(J_{ij}-J_{ji})[a_j,a_i]\,,
\end{equation*}
and
\begin{equation*}
\hat R=\sum_{i}\bigl(J_{ii}-\tfrac12 h\bigr)a_i^2
+\sum_{i<j}\tfrac12(J_{ij}+J_{ji})(a_ja_i+a_ia_j)\,,
\end{equation*}
which are uncorrelated by direct inspection.

We henceforth assume $[a_i,a_j]=0$ for all $i,j$.
For the \emph{uniformly accurate order~$1$ integrator} 
we have to leading order (again see Appendix~A)
\begin{equation*}
R^\magnus=\sum_{i}\tfrac12(J_{i0}-J_{0i})[a_0,a_i]\,,
\end{equation*}
and 
\begin{align*}
\hat R=&\;\tfrac12h^2a_0^2+\sum_{i}\tfrac12\bigl(J_{i0}+J_{0i})(a_0a_i+a_ia_0)
+\tfrac14h^2(a_i^2a_0+a_0a_i^2)\\
&\;+\sum_{i,j,k}\bigl(J_{ijk}-\mathbb E(J_{ijk})\bigr)a_ka_ja_i\,,
\end{align*}
which again, by direct inspection are uncorrelated.

For the \emph{uniformly accurate order $3/2$ integrator} 
the local flow remainders are 
\begin{align*}
R^\neu=&\;\sum_{i}(J_{ii0}-\tfrac14h^2)a_0a_i^2
+J_{i0i}a_ia_0a_i+(J_{0ii}-\tfrac14h^2)a_i^2a_0\\
&\;+\sum_{i<j}\Bigl((J_{0ji}+J_{0ij})a_ia_ja_0+J_{j0i}a_ia_0a_j\\
&\;\qquad\qquad+J_{i0j}a_ja_0a_i+(J_{ji0}+J_{ij0})a_0a_ia_j\Bigr)\\
&\;+\sum_{i,j,k,l}\bigl(J_{ijkl}-\mathbb E(J_{ijkl})\bigr)a_la_ka_ja_i\,,
\end{align*} 
and 
\begin{align*}
R^\magnus=&\;\sum_{i}\tfrac{1}{12}(J_i^2J_0-h^2-6J_{i0i})[a_i,[a_i,a_0]]\\
&\;+\sum_{i<j}\tfrac16\bigl(J_0J_iJ_j-3(J_{i0j}+J_{j0i})\bigr)
[a_j,[a_i,a_0]]\,.
\end{align*} 
Consequently we have
\begin{align*}
\hat R=&\;\tfrac{1}{12}\sum_{i}\Bigl((6J_{i0}J_i-J_i^2J_0-2h^2)a_0a_i^2
+2(J_i^2J_0-h^2)a_ia_0a_i\\
&\;\qquad\qquad+(6J_{0i}J_i-J_i^2J_0-2h^2)a_i^2a_0\Bigr)\\
&\;+\tfrac16\sum_{i<j}\biggl(\bigl(3(J_iJ_{0j}+J_jJ_{0i})-J_0J_iJ_j\bigr)a_ia_ja_0\\
&\;\qquad\qquad+\bigl(J_0J_iJ_j+3(J_{i0j}-J_{j0i})\bigr)a_ja_0a_i\\
&\;\qquad\qquad+\bigl(J_0J_iJ_j+3(J_{j0i}-J_{i0j})\bigr)a_ia_0a_j\\
&\;\qquad\qquad+\bigl(3(J_iJ_{j0}+J_jJ_{i0})-J_0J_iJ_j\bigr)a_0a_ia_j\biggr)\\
&\;+\sum_{i,j,k,l}\bigl(J_{ijkl}-\mathbb E(J_{ijkl})\bigr)a_la_ka_ja_i\,.
\end{align*} 
First we note that the terms in $\hat R$ of the form
$J_{ijkl}a_la_ka_ja_i$, for which at least three of the 
indices distinct, are uncorrelated with any terms
in $R^\magnus$. We thus focus on the terms of this
form with at most two distinct indices, namely
\begin{equation*}
\sum_{i<j}(J_{ii}J_{jj}-\tfrac14h^2)a_j^2a_i^2
+\sum_{i\neq j}J_{iii}J_ja_ja_i^3
+\sum_{i}(J_{iiii}-\tfrac18h^2)a_i^4
\end{equation*}
and other remaining terms in $\hat R$. Since
$\mathbb E\bigl[J_{i0i}|J_i\bigr]\equiv h(J_i^2-h)/6$ 
and $\mathbb E\bigl[J_{i0j}|J_i,J_j\bigr]\equiv h J_i J_j/6$ 
for $i\neq j$ the following conditional expectations 
are immediate
\begin{align*}
\mathbb E\bigl[(J_i^2J_0-h^2)(J_i^2J_0-h^2-6J_{i0i})|J_i\bigr]=&\;0\,,\\
\mathbb E\bigl[(J_{iiii}-\tfrac18h^2)(J_i^2J_0-h^2-6J_{i0i})|J_i\bigr]
=&\;0\,,\\
\mathbb E\bigl[(J_{ii}J_{jj}-\tfrac14h^2)(J_i^2J_0-h^2-6J_{i0i})|J_i,J_j\bigr]
=&\;0\,,\\
\mathbb E\bigl[\bigl(J_0J_iJ_j+3(J_{i0j}-J_{j0i})\bigr)
\bigl(J_0J_iJ_j-3(J_{i0j}+J_{j0i})\bigr)|J_i,J_j\bigr]=&\;0\,,\\
\mathbb E\bigl[(J_{ii}J_{jj}-\tfrac14h^2)
\bigl(J_0J_iJ_j-3(J_{i0j}+J_{j0i})\bigr)|J_i,J_j\bigr]=&\;0\,,\\
\mathbb E\bigl[J_{iii}J_j\bigl(J_0J_iJ_j-3(J_{i0j}+J_{j0i})\bigr)|J_i,J_j\bigr]=&\;0\,.
\end{align*} 
Hence the expectations of the terms shown are also zero.
Secondly, direct computation of the following expectations
reveals
\begin{align*}
\mathbb E\bigl((6J_{0i}J_i-J_i^2J_0-2h^2)(J_i^2J_0-h^2-6J_{i0i})\bigr)
=&\;0\,,\\
\mathbb E\Bigl(\bigl(3(J_iJ_{0j}+J_jJ_{0i})-J_0J_iJ_j\bigr) 
\bigl(J_0J_iJ_j-3(J_{i0j}+J_{j0i})\bigr)\Bigr)=&\;0\,.
\end{align*}
Hence $R^\magnus$ and $\hat R$ are uncorrelated.

The corresponding global comparison results~\eqref{eq:globcom} 
now follow directly in each of the three cases above using
that the local errors accumulate and contribute to the
global error as diagonal terms in the standard manner described
in detail at the end of Section~\ref{sec:globalerror}. 
Note that we include the terms 
$\tfrac{1}{12}h^2[a_i,[a_i,a_0]]$ in the order~$1$ 
uniformly accurate Magnus integrator. 
These terms appear at leading order in the
global remainder where they would otherwise 
generate a non-positive definite contribution 
to $D_M$ in~\eqref{eq:dmexplicit}.
\hfill \qquad\end{proof}

Note that the Magnus integrator $\hat\sigma^{(1)}$ 
is an order~$1$ integrator 
without the terms $\tfrac{1}{12}h^2[a_i,[a_i,a_0]]$. 
However they are cheap to compute and including them
in the integrator guarantees 
global superior accuracy independent of the set
of governing coefficient matrices.

\section{Non-commuting linear diffusion vector fields}
What happens when the linear diffusion vector fields do
not commute, i.e.\ we have $[a_i,a_j]\neq0$
for non-zero indices $i\neq j$? Hereafter assume
$i,j,k\in\{1,\ldots,d\}$.
Consider the case of order~$1$ integrators. 
The local flow remainders are 
\begin{equation*}
R^\neu=\sum_{i}(J_{i0}a_0a_i+J_{0i}a_ia_0)
+\sum_{i,j,k}J_{kji}a_ia_ja_k
\end{equation*} 
and
\begin{align*}
R^\magnus=&\;\sum_{i}\tfrac12(J_{i0}-J_{0i})[a_0,a_i]
+\sum_{i\neq j}\bigl(J_{iij}-\tfrac{1}{2}J_iJ_{ij}
+\tfrac{1}{12}J_i^2J_j\bigr)[a_i,[a_i,a_j]]\\
&\;+\sum_{i<j<k} \bigl((J_{ijk}+\tfrac12 J_jJ_{ki}+\tfrac12 J_kJ_{ij}
-\tfrac23 J_iJ_jJ_k)[a_i,[a_j,a_k]]\\
&\;\qquad\qquad+(J_{jik}+\tfrac12 J_iJ_{kj}+\tfrac12 J_kJ_{ji}
-\tfrac23 J_iJ_jJ_k)[a_j,[a_i,a_k]]\bigr)\,.
\end{align*}
Computing $D_M$ in~\eqref{eq:dmexplicit} gives
\begin{equation*}
D_M=h^{2}\sum_{i\neq j}U_{iij}^{\text{T}}BU_{iij}
+h^{2}\sum_{i\neq j\neq k}U_{ijk}^{\text{T}}CU_{ijk}
\end{equation*} 
where
$U_{iij}=(a_ja_i^2,a_ia_ja_i,
a_i^2a_j,a_j^3,a_0a_j,a_ja_0)^\tr\in\mathbb R^{6p\times p}$
and in addition we have that
$U_{ijk}=(a_ka_ja_i,a_ja_ka_i,
a_ka_ia_j,a_ia_ka_j,a_ja_ia_k,a_ia_ja_k)^\tr
\in\mathbb R^{6p\times p}$.
Here $B,C\in\mathbb R^{6p\times 6p}$ 
consist of $p\times p$ diagonal blocks of
the form $b_{k\ell}I_p$ and $c_{k\ell}I_p$
where
\begin{equation*}
b=\tfrac{1}{144}\begin{pmatrix}  31 & 10 & 1 & 18 & 12 & 24\\
                                 10 & 4 & 10 & 0  & 0  & 0 \\
                                 1 & 10 & 31 & 18 & 24 & 12\\
                                 18& 0  & 18 & 60 & 36 & 36\\
                                 12& 0  & 24 & 36 & 36 & 36\\
                                 24& 0  & 12 & 36 & 36 & 36
\end{pmatrix}\,,
\end{equation*}
and
\begin{equation*}
c=\tfrac{1}{36}\begin{pmatrix}  4 & 1 & 1 & 1 & 1 &-2\\
                                1 & 4 & 1 & -2& 1 & 1\\
                                1 & 1 & 4 & 1 & -2& 1\\
                                1 & -2& 1 & 4 & 1 & 1\\
                                1 & 1 &-2 & 1 & 4 & 1\\
                                -2& 1 & 1 & 1 & 1 & 4
\end{pmatrix}\,.
\end{equation*}
Again, $c$ is positive semi-definite with eigenvalues 
$\{\tfrac16,\tfrac16,\tfrac16,\tfrac16,0,0\}$. However $b$
has eigenvalues 
$\{\tfrac16,\tfrac{1}{48}(5+\sqrt 41),\tfrac{1}{48}(5-\sqrt 41),
0.94465,0.0943205,-0.03897\}$, where the final three values are
approximations to the roots of $288x^3-2888x^2+14x+1$.

The eigenvalues of $b$ and $c$, respectively, 
are multiple eigenvalues for the matrices $B$ 
and $C$, respectively. This implies that there 
are certain matrix combinations and initial
conditions, for which the order~$1$ Taylor approximation 
is more accurate in the mean-square sense
than the Magnus approximation. However, the two 
negative values are small in absolute value
compared to the positive eigenvalues. For the 
majority of systems, one can thus expect the
Magnus aproximation to be more accurate
(as already observed by Sipil\"ainen~\cite{Si}, 
Burrage~\cite{B} and Burrage and Burrage~\cite{BB}).
For any given linear system of stochastic differential
equations, the scheme that is more accurate can 
be identified using the results above.

In this case there are terms from 
$\tfrac12(\hat\sigma\rho+\rho\hat\sigma)$ 
in Magnus remainder expression~\eqref{origmer}
that appear at leading order in the local flow 
remainder. These terms are of the form 
$\tfrac{1}{12}\bigl(a_i[a_j,[a_j,a_i]]+a_j[a_i,[a_i,a_j]]\bigr)h^2$.
They make a negative definite contribution to global error,
though this can be negated by including
them as cheaply computable terms in the Magnus integrator
(indeed we recommend doing so).

\section{Quadrature}
\label{sec:quadrature}
We start by emphasizing that there are two inherent scales: 
\begin{itemize}
\item \emph{Quadrature scale} $\Delta t$ 
on which the discrete Wiener paths are generated;
\item \emph{Evaluation scale} $h$ on which
the stochastic differential equation is advanced.
\end{itemize}
The idea is to approximate multiple Stratonovich
integrals by their corresponding expectations
conditioned on the $\sigma$-algebra representing 
intervening knowledge of the Wiener paths
(Clark and Cameron~\cite{CC}; Newton~\cite{N}; 
Gaines and Lyons~\cite{GL97}). 
Hence we approximate $J_\alpha(t_n,t_{n+1})$ by
$\bE\,\bigl(\left. J_\alpha(t_n,t_{n+1})\right|\mathcal F_Q\bigr)$ 
where
\begin{equation*}
{\mathcal F}_Q=\{\Delta W^i_{t_n+q\,\Delta t}\colon 
i=1,\ldots,d;\, q=0,\ldots,Q-1;\, n=0,\dots,N-1\}\,,
\end{equation*}
with $\Delta W^i_{t_n+q\,\Delta t}\equiv 
W^i_{t_n+(q+1)\,\Delta t}-W^i_{t_n+q\,\Delta t}$,
and $Q\Delta t\equiv h$, i.e.\ $Q$ is the number 
of Wiener increments.
We extend the result of Clark and Cameron~\cite{CC} 
on the maximum rate of convergence to arbitrary
order multiple Stratonovich integrals.

\begin{lemma}[Quadrature error]\label{lem:quaderr}
Suppose at least two of the indices in the 
multi-index $\alpha=\{\alpha_1,\cdots,\alpha_\ell\}$ are distinct.
Define 
\begin{equation*}
j^*\equiv\min\{j\le\ell-1\colon \alpha_i=\alpha_{l}\,,\,\forall\, i,l>j\}\,.
\end{equation*}
Let $\nn(\alpha)$ denote the number of zeros in $\alpha$,
and $\nn^*(\alpha)=\nn(\{\alpha_{j^*},\ldots,\alpha_\ell\})$.
Then the $L^2$ error in the multiple Stratonovich integral 
approximation 
$\bE\,\bigl(\left. J_\alpha(t_n,t_{n+1})\right|\mathcal F_Q\bigr)$ is
\begin{equation}  \label{L2errmeas}
\bigl\|J_\alpha(t_n,t_{n+1})-
\bE\,\bigl(\left. J_\alpha(t_n,t_{n+1})\right|\mathcal F_Q\bigr)\bigr\|_{L^2}
=\mathcal
O\biggl(\frac{h^{(\ell+\nn(\alpha))/2}}{Q^{(\nn^*(\alpha)+1)/2}}\biggr)\,.
\end{equation}  
\end{lemma}
 
\begin{proof} 
For any $s, t$, we write for brevity
\begin{equation*}
\widehat{J}_\alpha(s,t)=\bE\,\bigl(J_{\alpha}(s,t)\, |\mathcal F_Q\bigr)\,.
\end{equation*} 
We define $\alpha^{-k}$ as the multi-index obtained by deleting the last
$k$ indices, that is  $\alpha^{-k}=\{\alpha_1,\cdots,\alpha_{\ell-k}\}$.
We set $\tau_q\equiv t_n+q\Delta t$,
where $q=0,\ldots,Q-1$. Then
\begin{align*}
J_{\alpha}(t_n,t_{n+1}) 
&=\sum_{q=0}^{Q-1}\int_{\tau_q}^{\tau_{q+1}}
J_{\alpha^{-1}}(t_n,\tau)\,\mathrm{d}W_\tau^{\alpha_\ell}\\
&=\sum_{q=0}^{Q-1}\Biggl(
\sum_{k=1}^{\ell-1}J_{\alpha^{-k}}(t_n,\tau_q)
J_{\alpha_{\ell-k+1},\ldots,\alpha_\ell}(\tau_q,\tau_{q+1})
+J_{\alpha}(\tau_q,\tau_{q+1})\Biggr)\,.
\end{align*}
Thus we have
\begin{align}
\lefteqn{J_\alpha(t_n,t_{n+1})-
\widehat{J}_\alpha(t_n,t_{n+1})}\nonumber \\ 
& \qquad = \sum_{q=0}^{Q-1}\Biggl(\sum_{k=1}^{\ell-1}\bigl(J_{\alpha^{-k}}(t_n,\tau_q) 
-\widehat{J}_{\alpha^{-k}}(t_n,\tau_q)\bigr)
J_{\alpha_{\ell-k+1},\ldots,\alpha_\ell}(\tau_q,\tau_{q+1})\nonumber \\
& \qquad\qquad\qquad +\widehat{J}_{\alpha^{-k}}(t_n,\tau_q)\, 
\bigl(J_{\alpha_{\ell-k+1},\ldots,\alpha_\ell}(\tau_q,\tau_{q+1})
-\widehat{J}_{\alpha_{\ell-k+1},\ldots,\alpha_\ell}(\tau_q,\tau_{q+1})\bigr)\nonumber\\
& \qquad\qquad\qquad +J_{\alpha}(\tau_q,\tau_{q+1})
-\widehat{J}_{\alpha}(\tau_q,\tau_{q+1})\Biggr)\,.
\label{integral_decomp}
\end{align} 

We prove the assertion by induction over $\ell$.
For $\ell=2$, the first two terms 
in the sum~\eqref{integral_decomp} are zero and 
\begin{align*}
\biggl\|\sum_{q=0}^{Q-1}\bigl(J_{\alpha}(\tau_q,\tau_{q+1})
-\widehat{J}_{\alpha}(\tau_q,\tau_{q+1})\bigr)\biggr\|_2^2 & 
=\mathcal O\bigl(Q(\Delta t)^{\ell+\nn(\alpha)}\bigr)\\
&=\mathcal O\biggl(
\frac{h^{\ell+\nn(\alpha)}}{Q^{\ell+\nn(\alpha)-1}}\biggr)\\
&=\mathcal O\biggl(
 \frac{h^{\ell+\nn(\alpha)}}{Q^{\nn^*(\alpha)+1}}\biggr)\,.
\end{align*} 
Here we have used that fact that $\ell=2$,
$j^*=1$ and thus $\nn(\alpha)=\nn^*(\alpha)$.

Assume now that $\ell >2$. We will investigate the order of each of
the three types of terms in \eqref{integral_decomp} separately.
If at least two indices in $\alpha^{-k}$ are distinct, 
then by induction hypothesis
\begin{equation*}
\bigl\|J_{\alpha^{-k}}(t_n,\tau_q) -
\widehat{J}_{\alpha^{-k}}(t_n,\tau_q)\bigr\|_2^2
 = \mathcal O\biggl(
\frac{(\tau_q-t_n)^{\ell-k+\nn(\alpha^{-k})}}{q^{\nn^*(\alpha^{-k})+1}}
\biggr)\,,
\end{equation*}
and since $\tau_q-t_n=q\Delta t$ and $Q\,\Delta t=h$, we have
for each $k=1,\ldots,\ell-1$,
\begin{align*}
\sum_{q=0}^{Q-1}\bigl\|\bigl(J_{\alpha^{-k}}(t_n,\tau_q) 
-\widehat{J}_{\alpha^{-k}}(t_n,\tau_q)\bigr)&\,
J_{\alpha_{\ell-k+1},\ldots,\alpha_\ell)}(\tau_q,\tau_{q+1})\bigr\|_2^2\\
&=\mathcal
O\biggl(\frac{h^{\ell+\nn(\alpha)}}{Q^{\nn^*(\alpha^{-k})+\nn(\alpha)-\nn(\alpha^{-k})+k}}
\biggr)\,.
\end{align*}
Note that we have
\begin{equation}\label{estimate1}
\nn^*(\alpha^{-k})+\nn(\alpha)-\nn(\alpha^{-k})+k
\geq\nn^*(\alpha)+k\geq\nn^*(\alpha)+1\,.
\end{equation} 
If all indices in $\alpha^{-k}$ are equal, then 
$\bigl\|J_{\alpha^{-k}}(t_n,\tau_q)-
\widehat{J}_{\alpha^{-k}}(t_n,\tau_q)\bigr\|_2^2=0$.

For the second term in \eqref{integral_decomp} we have  
for each $k=1,\ldots,\ell-1$,
\begin{align*}
\sum_{q=0}^{Q-1}\bigl\|\widehat{J}_{\alpha^{-k}}(t_n,\tau_q)\, 
\bigl(&J_{\alpha_{\ell-k+1},\ldots,\alpha_\ell}(\tau_q,\tau_{q+1})
-\widehat{J}_{\alpha_{\ell-k+1},\ldots,\alpha_\ell}(\tau_q,\tau_{q+1})\bigr)
\bigr\|_2^2 \\
& = \begin{cases} 0\,,& \text{if}~\alpha_\ell=\ldots=\alpha_{\ell-k+1}\,,\\
\mathcal O\biggl(\frac{h^{\ell+\nn(\alpha)}}{Q^{\nn(\alpha)-\nn(\alpha^{-k})+k-1}}\biggr)
\,, & \text{otherwise}\,.
\end{cases}
\end{align*}
Note that in the second of these cases 
\begin{equation}\label{estimate3}
\nn(\alpha)-\nn(\alpha^{-k})+k-1\geq\nn^*(\alpha)+1\,.
\end{equation}

For the third term we have 
\begin{equation*}
\sum_{q=0}^{Q-1}\bigl\|J_{\alpha}(\tau_q,\tau_{q+1})
-\widehat{J}_{\alpha}(\tau_q,\tau_{q+1})\bigr\|_2^2=
\mathcal O\biggl(\frac{h^{\ell+\nn(\alpha)}}{Q^{\ell+\nn(\alpha)-1}}\biggr)\,.
\end{equation*}
Again, note that we have
\begin{equation}\label{estimate4}
\ell+\nn(\alpha)-1\geq\nn^*(\alpha)+1\,.
\end{equation}

Equality holds in at least one of \eqref{estimate1} 
to \eqref{estimate4}. To see this we distinguish the case
when the last two indices are equal, 
$\alpha_{\ell-1}=\alpha_\ell$, and the case 
when the last two indices are distinct,
$\alpha_{\ell-1}\neq\alpha_\ell$.
If $\alpha_{\ell-1}=\alpha_\ell$, then
\begin{equation*}
\nn^*(\alpha^{-1})+\nn(\alpha)-\nn(\alpha^{-1})=\nn^*(\alpha)\,.
\end{equation*}
Since in this case at least two indices in
$\alpha^{-1}$ are distinct, 
equality holds for $k=1$ in~\eqref{estimate1}.
If $\alpha_{\ell-1}\not=\alpha_\ell$, then 
\begin{equation*}
\nn^*(\alpha)=\nn(\alpha)-\nn(\alpha^{-2})\,,
\end{equation*}
and thus equality holds for $k=2$ 
in~\eqref{estimate3}. Hence the lemma follows.\hfill \qquad\end{proof}

Note that each multiple Stratonovich integral 
$J_\alpha(t_n,t_{n+1})$ can be thought of as an $\ell$-dimensional
topologically conical volume in $(W^1,\ldots,W^\ell)$-space.
The surface of the conical volume is \emph{panelled} 
with each panel distinguished by a double Stratonovich 
integral term involving two consecutive indices from $\alpha$.
The edges between the panels are distinguished by a triple
integral and so forth. The conditional expectation approximation
$\bE\bigl(\left. J_{\alpha}(t_n,t_{n+1})\right|\mathcal F_Q\bigr)$
can also be decomposed in this way. In the $L^2$-error
estimate for this approximation, the leading terms are given
by sums over the panels which also confirm the 
estimate~\eqref{L2errmeas}. 

Approximations of multiple Stratonovich 
integrals constructed using their conditional 
expectations are intimately linked to 
those constructed using paths $W^i_t$ that 
are approximated by piecewise linear interpolations of the intervening
sample points. The difference between the two approaches 
are asymptotically smaller terms. For more
details see Wong and Zakai~\cite{WZ}, 
Kloeden and Platen~\cite{KP}, 
Hofmann and M\"uller-Gronbach~\cite{HM}
and Gy\"ongy and Michaletzky~\cite{GM}.

\section{Global error vs computational effort}\label{sec:quadeff}
We examine in detail the stepsize/accuracy 
regimes for which higher order stochastic integrators are
feasible and also when they become less efficient than
lower order schemes.
In any strong simulation there are two principle sources
of computational effort. Firstly there is evaluation effort
$\mU^\eval$ associated with evaluating the 
vector fields, their compositions and
any functions such as the matrix exponential.
Secondly there is the quadrature effort\/ $\mU^\quadr$ 
associated with approximating multiple stochastic integrals
to an accuracy commensurate with the order of
the method. For a numerical approximation of order $M$ 
the computational evaluation effort measured
in flops over $N=Th^{-1}$ evaluation steps is 
\begin{equation*}
\mU^\eval=(c_Mp^2+c_E)\,Th^{-1}\,.
\end{equation*}
Here $p$ is the size of the system, $c_M$ represents
the number of scalar-matrix multiplications and matrix-matrix
additions for the order $M$ truncated Magnus expansion,  
and $c_E$ is the effort required to
compute the matrix exponential.
Note that if we implement an order $1/2$ method there is no
quadrature effort. Hence since $\mE=\mathcal O(h^{1/2})$
we have $\mE=\mathcal O\bigl((\mU^\eval)^{-1/2}\bigr)$. 

Suppose we are required to simulate 
$J_{\alpha_1\cdots \alpha_\dpp}(t_n,t_{n+1})$ with all the 
indices distinct with a global error of order $h^M$;
naturally $\dpp\geq2$ and $M\geq\dpp/2$.

\begin{lemma}\label{lem:quadscaling}
The quadrature effort\/ $\mU$ measured in flops required to approximate
$J_{\alpha_1\cdots\alpha_\dpp}(t_n,t_{n+1})$ with a global error of order $h^M$
when all the indices are distinct and non-zero, is to leading order in $h$: 
\begin{equation*}
\mU=\mathcal O\bigl(h^{-\beta(M,\dpp)}\bigr)\,,
\end{equation*}
where 
\begin{equation*}
\beta(M,\dpp)=(\dpp-1)(2M+1-\dpp)+1\,.
\end{equation*}
Since we stipulate the global error associated with
the multiple integral approximation to be
$\mE=\mathcal O(h^M)$, we have
\begin{equation*}
\mE=\mathcal O\bigl(\mU^{-M/\beta(M,\dpp)}\bigr)\,.
\end{equation*}
\end{lemma}

\begin{proof}
The quadrature effort required to construct  
$\bE \bigl(\left. J_{\alpha_1\cdots\alpha_\dpp}(t_n,t_{n+1})\right|\mathcal F_Q\bigr)$,
which is a $(\dpp-1)$-multiple sum, over $[0,T]$ is 
$\mU=\mathcal O(Q^{\dpp-1}N)$ with $N=Th^{-1}$.
Using Lemma~\ref{lem:quaderr} to achieve a 
global accuracy of order $h^M$ and therefore
local $L^2$-norm of order $h^{M+1/2}$ 
for this integral, requires that $Q=h^{\dpp-1-2M}$.
\hfill \qquad \end{proof}

\begin{table}
\caption{Slopes of the logarithm of the global error $\mE$ 
verses the logarithm of the quadrature effort\/ $\mU$,
i.e.\ the exponent $-M/\beta(M,\dpp)$,
for different values of $\dpp$ and $M$ when 
$J_{\alpha_1\cdots\alpha_\dpp}$ has distinct non-zero indices.}
\begin{center}\footnotesize
\begin{tabular}{|c|ccccc|} \hline\hline
        & $\dpp=2$ & $\dpp=3$ & $\dpp=4$ & $\dpp=5$ & $\dpp=6$\\\hline
no zero index    &   &    &   &  &  \\ \hline
$M=1$    & $-1/2$  & $\cdots$ & $\cdots$ & $\cdots$ &  $\cdots$ \\
$M=3/2$  & $-1/2$  & $-1/2$   & $\cdots$ & $\cdots$ &  $\cdots$ \\
$M=2$    & $-1/2$  & $-2/5$   & $-1/2$   & $\cdots$ &  $\cdots$ \\
$M=5/2$  & $-1/2$  & $-5/14$  & $-5/14$  & $-1/2$   &  $\cdots$ \\
$M=3$    & $-1/2$  & $-1/3$   & $-3/10$  & $-1/3$   &  $-1/2$   \\
\hline\hline
\end{tabular}
\end{center} 
\end{table}

In Table~7.1 we quote values for the exponent 
$-M/\beta(M,\dpp)$ for different values of 
$M$ and $\dpp$ in the case when all the distinct indices 
$\alpha_1,\ldots,\alpha_\dpp$ are non-zero.

Suppose we are given a stochastic differential equation 
driven by a $d$-dimensional Wiener process with 
non-commuting governing vector fields. To successfully
implement a strong numerical method of order $M$
we must guarantee that the global error associated 
with each multiple integral present in the integrator
that is approximated by its conditional expectation
is also of order $M$. If we 
implement a numerical method of order $M\leq d/2$,
we will \emph{in general} be required to simulate
multiple Stratonovich integrals with distinct
indices of length $\dpp$ with $2\leq\dpp\leq 2M\leq d$. 
We will also have to simulate multiple integrals with
repeated indices of length $\dpp_r\leq 2M$. These integrals 
will require the same or fewer quadrature points than 
those with distinct indices as we can take advantage
of the repeated indices---see Section~\ref{sec:effquadbasis} 
for more details. Such integrals therefore
represent lower order corrections to the 
quadrature effort. Similarly multiple integrals 
with distinct indices that involve a zero index 
have index length
that is one less than similar order multiple integrals 
with distinct non-zero indices. Hence they
will also represent lower order corrections to the 
quadrature effort.

To examine the scaling exponent in the relation
between the global error $\mE$ and the quadrature effort 
$\mU^\quadr$, which is the sum total of the efforts
required to approximate all the required multiple
integrals to order $h^M$, we use Table~7.1 as a
guide for the dominant scalings. 
For methods of order $M\leq d/2$, 
if $d=2$ and we implement a method of order $M=1$,
then the dominant exponent is $-1/2$.
Similarly if $d=3$ then order $1$ and
$3/2$ methods also invoke a dominant 
scaling exponent of $-1/2$ for the integrals of 
length $\dpp=2$ and $\dpp=3$.
If $d=4$ then methods of order $1$ and $3/2$ have
the same scaling exponent of $-1/2$, however the
method of order $2$ involves multiple integrals
with three indices which are all distinct and 
the dominant scaling exponent for them is $-2/5$.

If we implement a method of order $M>d/2$
then we will be required to simulate
multiple Stratonovich integrals with distinct
indices of length $\dpp$ with $2\leq\dpp\leq d$.
We must also simulate higher order multiple integrals 
with indices of length $\dpp_r$ involving repetitions
with $d\leq\dpp_r\leq2M$; these may be 
cheaper to simulate than multiple integrals of
the same length with distinct indices (again see
Section~\ref{sec:effquadbasis}). 
When $d=2$ the dominant scaling exponent is
$-1/2$ for all orders. For $d\geq3$ the 
dominant scaling exponent is \emph{at best} $-1/2$,
and so forth.
 
Lastly, we give an estimate for the critical stepsize
$h_{\mathrm{cr}}$ below which the quadrature effort
dominates the evaluation effort. Since $\dpp=M+1$
minimizes $\beta(M,\dpp)$ we have the following estimate.

\begin{corollary}\label{lem:hcr}
For the case of general non-commuting governing
vector fields and a numerical approximation of order $M$,
we have $\mU^\eval\geq \mU^\quadr$
if and only if $h\geq h_{\mathrm{cr}}$ where the critical stepsize
\begin{equation*}
h_{\mathrm{cr}}=\mathcal O\Bigl(
\bigl(T(c_Mp^2+c_E)\bigr)^{-1/(1-\beta(M,\dpp_{\max}))}\Bigr)\,,
\end{equation*}
where $\dpp_{\max}=\max\{d,M+1\}$.
\end{corollary}

In practice when we implement numerical methods for
stochastic differential equations driven by a $d$-dimensional
Wiener process we expect that for $h\geq h_{\mathrm{cr}}$
the evaluation effort dominates the compuational cost.
In this scenario integrators of order $M$ scale like
their deterministic counterparts. Consider what we
might expect to see in a log-log plot of global error
verses computational cost.
As a function of increasing computational cost 
we expect the global error for each method to 
fan out with slope $-M$, with higher order methods 
providing superior accuracy for a given effort.
However once the quadrature effort starts to dominate, 
the scaling exponents described above take over. 
When $d=2$ for example and all methods dress
themselves with the scaling exponent $-1/2$,
then we expect to see parallel graphs with higher order
methods still providing superior accuracy for
a given cost. However higher order methods 
that assume a scaling exponent worse than $-1/2$ will
eventually re-intersect the graphs of 
their lower order counterparts and past 
that regime should not be used.

Note that in the case when all the diffusion vector fields commute,
methods of order~$1$ do not involve any quadrature effort 
and hence $\mE=\mathcal O\bigl((\mU^\eval)^{-1}\bigr)$.
Using Lemma~\ref{lem:quaderr}, we can by analogy with
the arguments in the proof of Lemma~\ref{lem:quadscaling},
determine the dominant scaling exponents for methods of 
order $M\geq 3/2$. For example the $L^2$-error associated
with approximating $J_{0i}$ by its expectation
conditioned on intervening information is 
of order $h^{3/2}/Q$. Hence we need only choose $Q=h^{-1/2}$
to achieve to achieve a global error of order $3/2$. 
In this case the dominant scaling exponent is $-1$.
However the $L^2$-error associated
with approximating $J_{0ij}$ for $i\neq j$
is of order $h^2/Q^{1/2}$ whereas for $J_{i0j}$
and $J_{ij0}$ it is of order $h^2/Q$.
For the case of diffusing vector fields we do not
need to simulate $J_{0ij}$, and so for a
method of order $2$ the dominant scaling exponent 
is still $-1$. However more generally the effort 
associated with approximating $J_{0ij}$ 
dominates the effort associated with the other 
two integrals.

\section{Efficient quadrature bases}\label{sec:effquadbasis}
When multiple Stratonovich integrals contain repeated indices, 
are they as cheap to compute as the corresponding
lower dimensional integrals with an equal number
of distinct indices (none of them repeated)?

Let $i\cdots i_p$ denote the multi-index with 
$p$ copies of the index $i$. Repeated integration 
by parts yields the formulae
\begin{subequations}\label{eq:parts}
\begin{align}
J_{i\cdots i_pji\cdots i_q}&=\sum_{k=1}^q(-1)^{k+1}
J_{i\cdots i_k}J_{i\cdots i_pji\cdots i_{q-k}}+
(-1)^{q+2}\int J_{i\cdots i_p}J_{i\cdots i_q}\mathrm{d}J_j\,,\label{eq:parts1}\\
J_{i\cdots i_pj\cdots j_q}&=\sum_{k=1}^{q-1}(-1)^{k+1}
J_{j\cdots j_k}J_{i\cdots i_pj\cdots j_{q-k}}
+(-1)^{q+1}\int J_{i\cdots i_p}\mathrm{d}J_{j\cdots j_q}\,.\label{eq:parts2}
\end{align}
\end{subequations}
The first relation~\eqref{eq:parts1} suggests that any
integral of the form $J_{i\cdots i_pji\cdots i_q}$ can always be
approximated by a single sum. This last statement is true
for $q=1$. If we assume it is true for $q-1$ and apply the
relation~\eqref{eq:parts1} we establish by induction that 
$J_{i\cdots i_pji\cdots i_q}$ can be approximated by a single sum.
A similar induction argument using~\eqref{eq:parts2} 
then also establishes that any integral 
of the form $J_{i\cdots i_pj\cdots j_q}$
can also be approximated by a single sum. Hence
in both cases the quadrature effort is proportional to $QN$. 

Implicit in the relations~\eqref{eq:parts} is the 
natural underlying shuffle algebra created by 
integration by parts (see Gaines~\cite{G,G2},
Kawksi~\cite{Ka} and Munthe--Kaas and Wright~\cite{MW}).
Two further results are of interest. Firstly
we remark that by integration by parts we have
the following two shuffle product results:
\begin{subequations}\label{eq:shuff}
\begin{align}
J_{i_1i_2i_3}J_{i_4}&=J_{i_1i_2i_3i_4}+J_{i_1i_2i_4i_3}
+J_{i_1i_4i_2i_3}+J_{i_4i_1i_2i_3}\,,\label{eq:shuff1}\\
J_{i_1i_2}J_{i_3i_4}&=J_{i_1i_2i_3i_4}+J_{i_1i_3i_2i_4}
+J_{i_3i_1i_2i_4}+J_{i_3i_4i_1i_2}
+J_{i_3i_1i_4i_2}+J_{i_1i_3i_4i_2}\,.\label{eq:shuff2}
\end{align}
\end{subequations}
If we replace $\{i_1,i_2,i_3,i_4\}$ by $\{i,i,j,j\}$ 
in~\eqref{eq:shuff2} and \eqref{eq:shuff1}
and then by $\{i,j,i,j\}$ in~\eqref{eq:shuff1} 
and \eqref{eq:shuff2},
respectively, we obtain the linear system of equations
\begin{equation}\label{eq:shufflin}
\begin{pmatrix} 
1&1&1&1 \\
0&2&0&1 \\
0&0&1&1 \\
0&0&0&2 
\end{pmatrix}
\begin{pmatrix} 
J_{jiji} \\
J_{ijji} \\
J_{jiij} \\
J_{ijij} 
\end{pmatrix}
=\begin{pmatrix} 
J_{ii}J_{jj}-J_{iijj}-J_{jjii} \\
J_{iji}J_{j}-J_{jjii}\\
J_{iij}J_{j}-2J_{iijj} \\
J_{ij}J_{ij}-4J_{iijj} 
\end{pmatrix}\,.
\end{equation}
By direct inspection the coefficient matrix on
the left-hand side has rank $4$ and so 
all the multiple Stratonovich
integrals $J_{jiji}$, $J_{ijji}$, $J_{jiij}$ 
and $J_{ijij}$ can be expressed in 
terms of $J_{iijj}$ and $J_{jjii}$ and products 
of lower order integrals, all of
which can be approximated by single sums
(note that $J_{iji}\equiv J_{ij}J_i-2J_{iij}$).

Now consider the set of multiple Stratonovich
integrals
\begin{equation*}
\mathfrak J=\bigl\{J_{i_1i_2i_3i_4i_5}
\colon \{i_1,i_2,i_3,i_4,i_5\}\in\mathrm{perms}\{i,i,i,j,j\}\bigr\}
\backslash \bigl\{J_{iiijj},J_{jjiii}\bigr\}\,,
\end{equation*}
where we exclude the elements $J_{iiijj}$ and $J_{jjiii}$ which 
we know can be approximated by single sums from~\eqref{eq:parts2}.
By considering the shuffle relations generated by products
of the form: $J_{i_1i_2i_3i_4}J_{i_5}$, $J_{i_1i_2i_3}J_{i_4i_5}$,
$J_{i_1i_2}J_{i_3i_4i_5}J_{i_5}$ and $J_{i_1}J_{i_2}J_{i_3}J_{i_4}J_{i_5}$
and substituting in the $10$ elements with indices from 
`$\mathrm{perms}\{i,i,i,j,j\}$'
we obtain an linear system of equations analogous to~\eqref{eq:shufflin}
with $50$ equations for the $8$ unknowns in $\mathfrak J$.
However direct calculation shows that the corresponding
coefficient matrix has rank $7$. In particular, all
of the multiple integrals in $\mathfrak J$ can 
be expressed in terms of $J_{iiijj}$, $J_{jjiii}$
and $J_{jijii}$.
Hence the set of multiple integrals with indices from 
`$\mathrm{perms}\{i,i,i,j,j\}$' cannot
all be approximated by single sums, but in fact require
a double sum to approximate $J_{jijii}$. 

For simplicity assume $d=1$. Consider 
numerical schemes of increasing order $M$. 
If $3/2\leq M\leq 3$
all the necessary multiple integrals can 
be approximated by single sums---at the 
highest order in this range indices 
involving permutations of $\{1,1,1,1,0\}$
and $\{1,1,0,0\}$ are included for which
the corresponding integrals 
can be approximated by single sums. For
methods of order $M\geq 7/2$ we require 
at least double sums to approximate the 
necessary multiple integrals.

When $d=2$, for methods of order $M=1,3/2$
the integrals involved can be approximated
by single sums, but for $M=2$ integrals
involving indices with permutations of $\{2,1,0\}$ 
are included which can only be approximated by double sums. 
If there were no drift vector
field then for $1\leq M\leq 2$ the 
necessary multiple integrals can all
be approximated by single sums, but
for $M=5/2$ we need to include multiple integrals
involving indices with permutations of $\{2,2,1,1,1\}$ 
which require approximation by double sums. 
We can in principle extend these results to 
higher values $M$, however methods of order
$M\geq d/2$ for $d\geq 3$ are not commonly implemented!

\section{Numerical simulations}\label{numerics}

\subsection{Riccati system}
Our first application is for stochastic
Riccati differential systems---some classes
of which can be reformulated as linear systems
(see Freiling~\cite{F} and Schiff and Shnider~\cite{SS}). 
Such systems arise in stochastic
linear-quadratic optimal control problems, for example, 
mean-variance hedging in finance 
(see Bobrovnytska and Schweizer~\cite{BS} and 
Kohlmann and Tang~\cite{KT})---though often 
these are backward problems (which we intend
to investigate in a separate study). 
Consider for example Riccati equations of the form
\begin{equation*}
u_t=u_0+\sum_{i=0}^d\int_0^t \bigl(u_\tau A_i(\tau)u_\tau
+B_i(\tau)u_\tau+u_\tau C_i(\tau)+D_i(\tau)\bigr)
\,\mathrm{d}W^i_\tau\,.
\end{equation*}
If $y=(U \,\, V)^\tr$ 
satisfies the linear stochastic differential system~\eqref{sde},
with
\begin{equation*}
a_i(t)\equiv\begin{pmatrix} B_i(t)& D_i(t)\\ -A_i(t) &
-C_i(t)\end{pmatrix}\,,
\end{equation*}
then $u=UV^{-1}$ solves the Riccati equation above.

We consider here a Riccati problem 
with two additive Wiener processes, $W^1$ and $W^2$,
and coefficient matrices 
\begin{equation}
D_0=\begin{pmatrix}\tfrac12 & \tfrac12 \\ 0 & 1\end{pmatrix}\,,
\qquad
D_1=\begin{pmatrix}0 & 1\\ -\tfrac12 & -\tfrac{51}{200}\end{pmatrix}
\qquad\text{and}\qquad
D_2=\begin{pmatrix}1 & 1\\ 1 & \tfrac12\end{pmatrix}\,,\label{mcoeffs}
\end{equation}
and
\begin{equation*}
A_0=\begin{pmatrix} -1 & 1 \\ -\tfrac12 & -1\end{pmatrix}\,,
\qquad\text{and}\qquad
C_0=\begin{pmatrix}-\tfrac12 & 0\\ -1 & -1\end{pmatrix}\,.
\end{equation*}
All other coefficient matrices are zero.
The initial data is the $2\times 2$ identity matrix, 
i.e.\ $u_0=I_2$ and therefore $U_0=I_2$ and $V_0=I_2$ also.
We found Higham~\cite{H} a very
useful starting point for our Matlab simulations.

Note that for this example the coefficient matrices $a_1$
and $a_2$ are upper right block triangular and
therefore nilpotent of degree $2$, and also that $a_1a_2$ and 
$a_2a_1$ are identically zero so that in particular
$[a_1,a_2]=0$. The number of terms in each integrator 
at either order~$1$ or $3/2$ is roughly equal, 
and so for a given stepsize the 
uniformly accurate Magnus integrators should 
be more expensive to compute due to the cost of computing 
the $4\times 4$ matrix exponential---we
used a $(6,6)$ Pad\'e approximation with scaling to
compute the matrix exponential. See Moler and Van Loan~\cite{MV}
and also Iserles and Zanna~\cite{IZ}, the computational
cost is roughly $6$ times the system size cubed.
Also note the order $1$ integrators do not 
involve quadrature effort whilst
the order $3/2$ integrators involve the quadrature
effort associated with approximating $J_{10}$ and $J_{20}$.
For comparison, we use a nonlinear Runge--Kutta type 
order $3/2$ scheme for the case of two additive noise terms 
(from Kloeden and Platen~\cite[p.~383]{KP}) applied directly
to the original Riccati equation:
\begin{align}
S_{t_n,t_{n+1}}=&\;S_{t_n}+f\bigl(S_{t_n}\bigr)h+D_1J_1+D_2J_2\notag\\
&\;+\tfrac{h}{4}\bigl(f(Y_1^+)+f(Y_1^-)
+f(Y_2^+)+f(Y_2^-)-4f\bigl(S_{t_n}\bigr)\bigr)\notag\\
&\;+\tfrac{1}{2\sqrt{h}}\left(\bigl(f(Y_1^+)-f(Y_1^-)\bigr)J_{10}
+\bigl(f(Y_2^+)-f(Y_2^-)\bigr)J_{20}\right)\,,\label{rk}
\end{align}
where $Y_j^{\pm}=S_{t_n}+\tfrac{h}{2}f\bigl(S_{t_n}\bigr)\pm D_j\sqrt{h}$ 
and $f(S)=SA_0S+B_0S+SC_0+D_0$.

In Figure~\ref{globalerrorriccati} we show how the global error 
scales with stepsize and also CPU clocktime for this Riccati problem. 
Note that as anticipated, for the same step size
(compare respective plot points starting from the left),
the order $1$ Magnus integrator is more expensive to compute 
and more accurate than the order $1$ Neumann integrator.
Now compare the order $3/2$ integrators. For the 
nonlinear scheme~\eqref{rk}, we must evaluate 
$f(S)$ five times per step per path costing 
$20p^3+54p^2$ flops---here 
$p=2$ refers to the size of the original Riccati system. 
For the Neumann and Magnus integrators the evaluation
costs are $16(2p\times p)=32p^2$ and 
$6(2p)^3+11(2p)^2=48p^3+44p^2$ flops, respectively 
(directly counting from the schemes).
Hence for large stepsize 
we expect the Neumann integrator to be cheapest
and the Magnus and nonlinear Runge--Kutta integrators
to be more expensive. 
However for much smaller stepsizes
the quadrature effort should start to dominate. 
The efforts of all the order $3/2$ integrators
will not be much different and the Magnus integrator then
outperforms the other two due to its superior accuracy.

\begin{figure}
  \begin{center}
  \includegraphics[width=9cm,height=5cm]{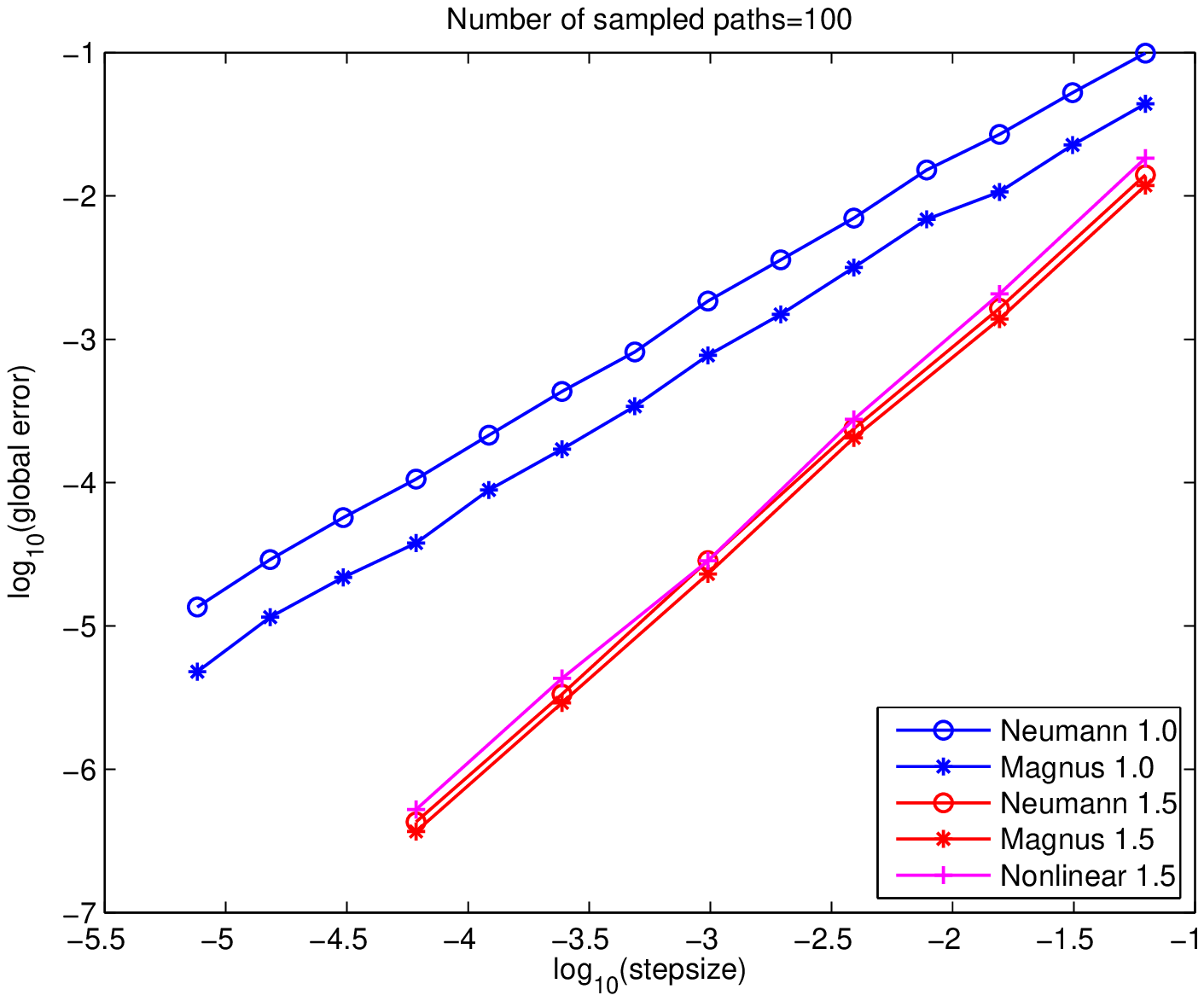}
  \includegraphics[width=9cm,height=5cm]{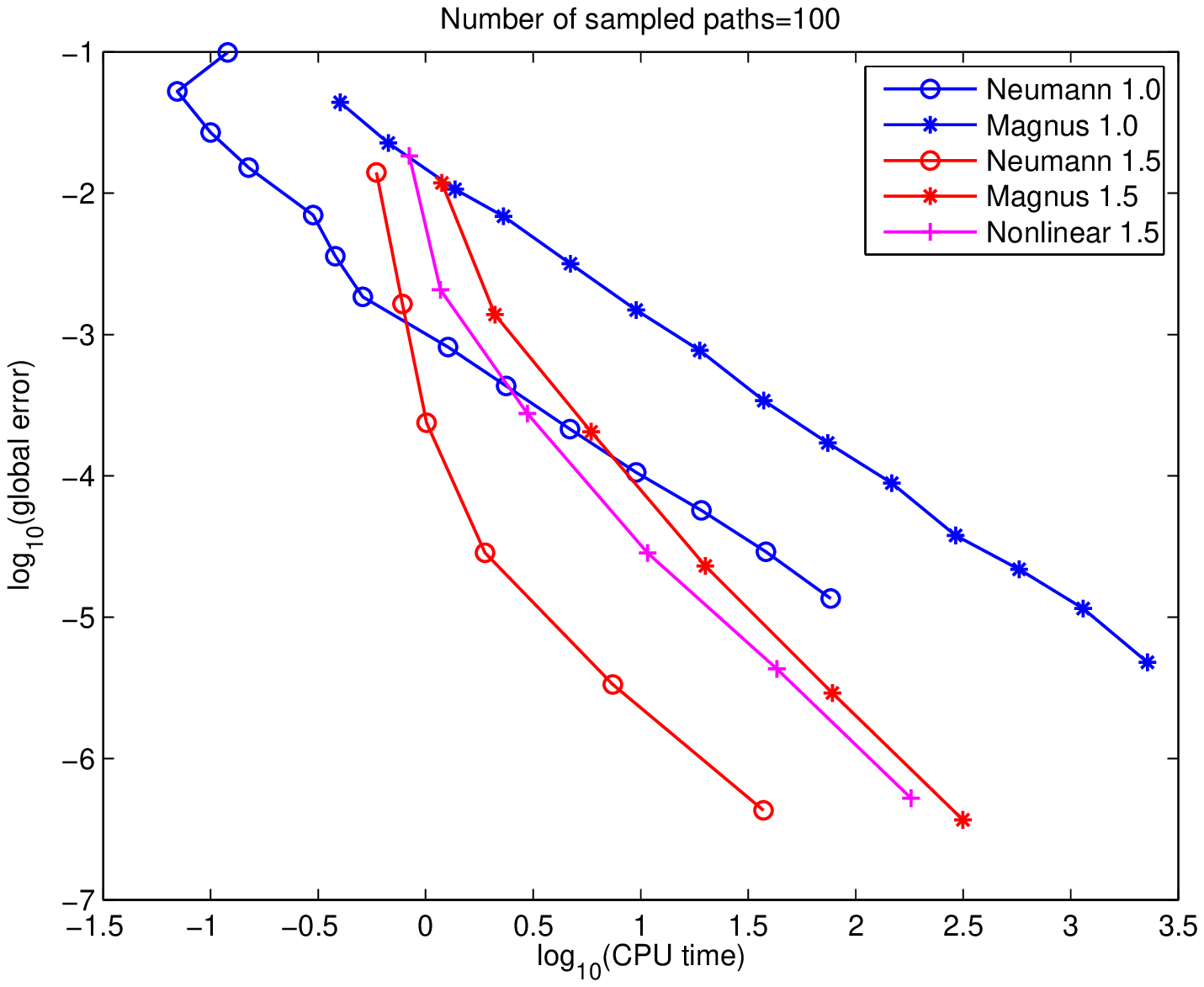}
  \end{center}
  \caption{Global error vs stepsize (top) 
  and vs CPU clocktime (bottom) 
  for the Riccati problem at time $t=1$. 
  The Magnus integrators
  of order~$1$ and $3/2$ shown are the 
  uniformly accurate Magnus integrators 
  from Section~\ref{globalcomparison}.}
  \label{globalerrorriccati}
\end{figure}

\subsection{Linear system}

\begin{figure}
  \begin{center}
  \includegraphics[width=9cm,height=5cm]{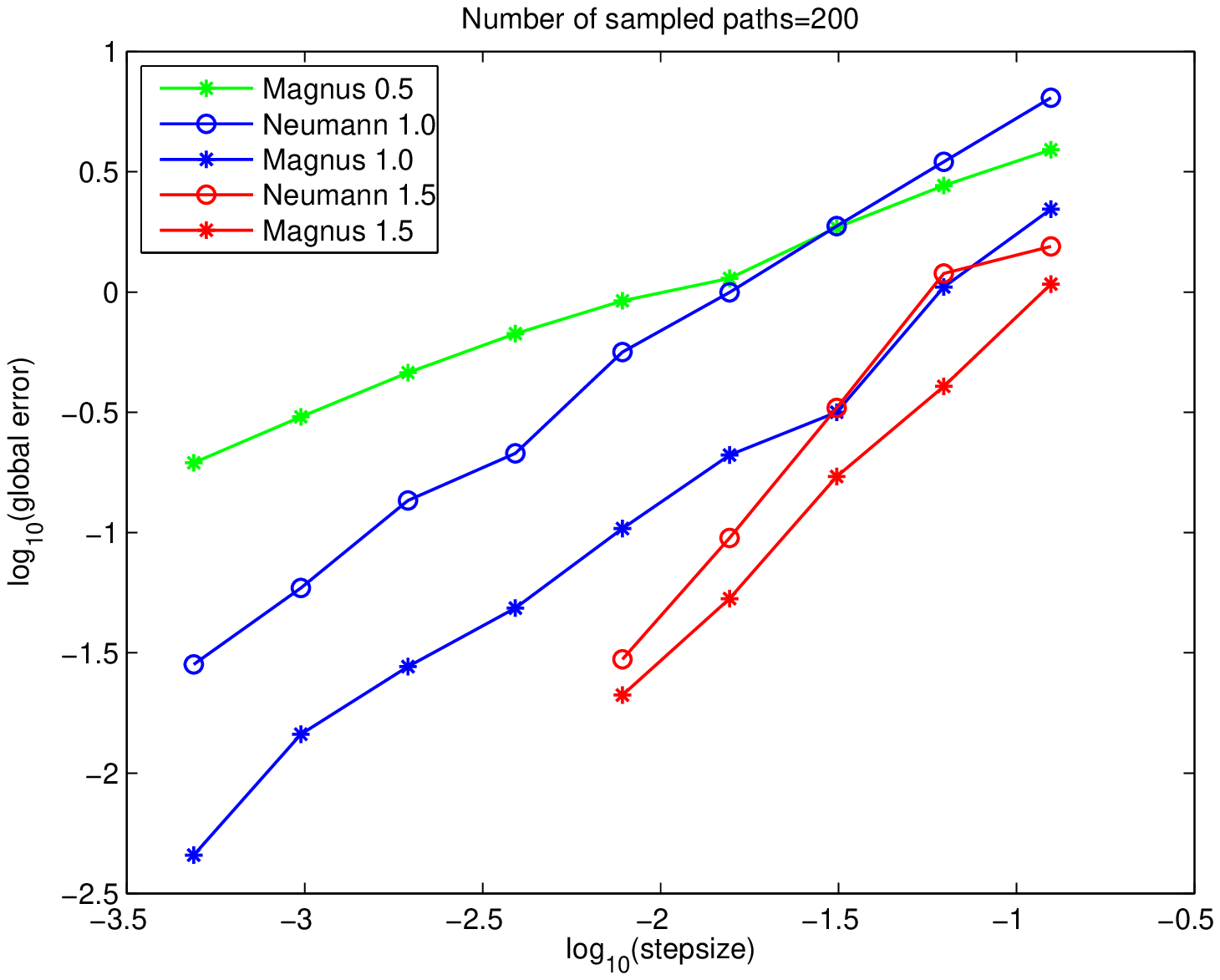}
  \includegraphics[width=9cm,height=5cm]{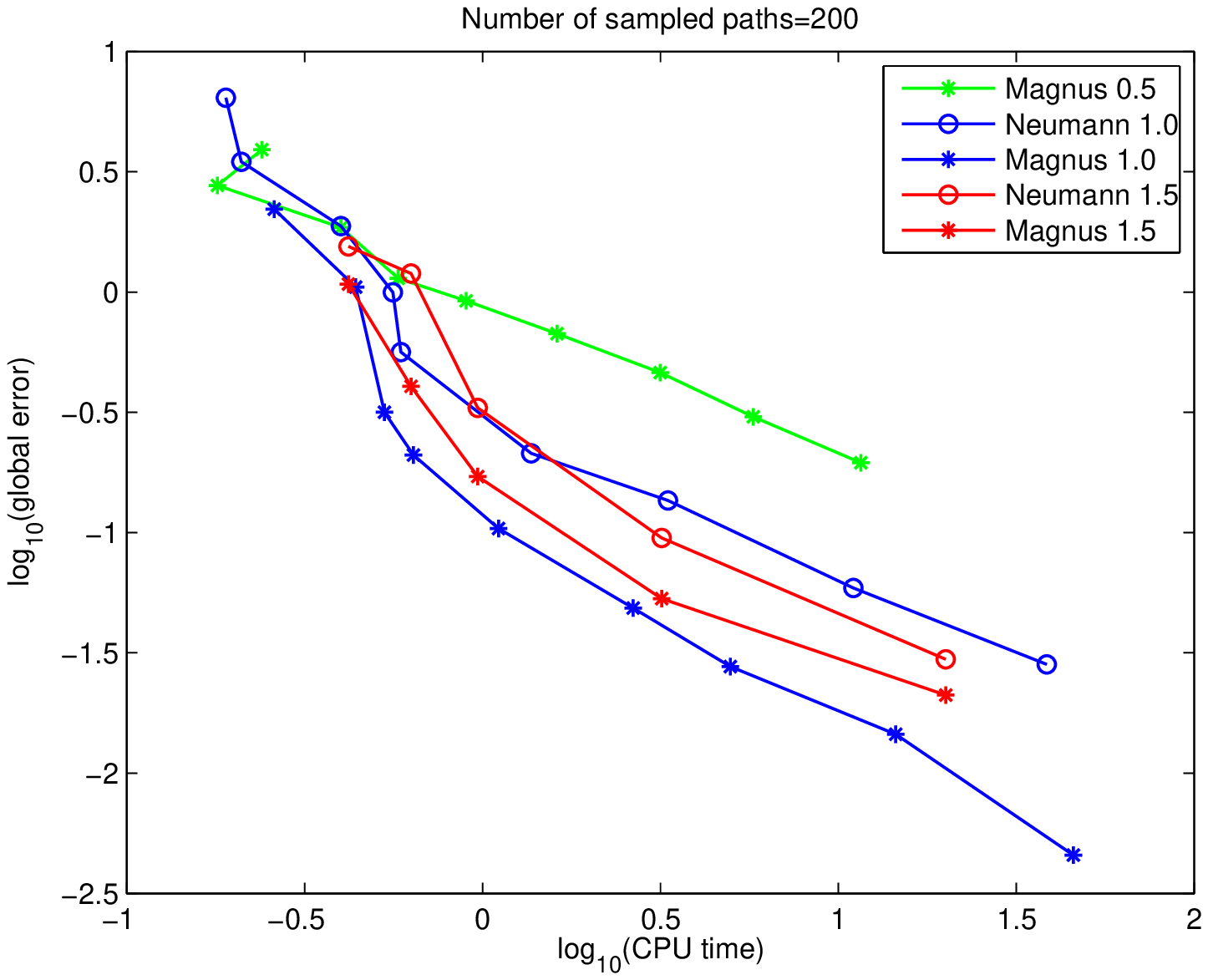}
  \end{center}
  \caption{Global error vs stepsize (top) 
  and vs CPU clocktime (bottom) 
  for the model problem at time $t=1$ with $2$ driving Wiener
  processes. The error corresponding to the 
  largest step size takes the shortest time to compute.}
  \label{globalerror}
\end{figure}

\begin{figure}
  \begin{center}
  \includegraphics[width=9cm,height=5cm]{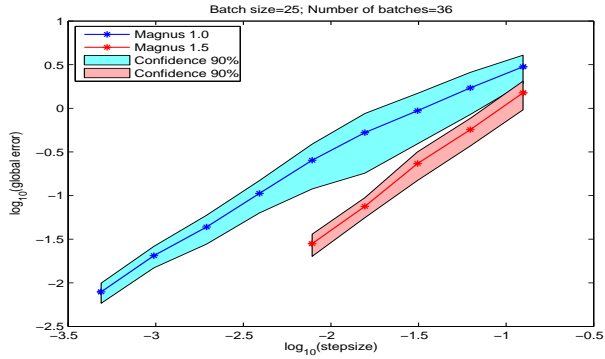}
  \end{center}
  \caption{Confidence intervals for the global errors of 
  the uniformly accurate Magnus integrators 
  for the model problem at time $t=1$ with $2$ driving Wiener
  processes.}
  \label{fig:confidence}
\end{figure}

\begin{figure}
  \begin{center}
  \includegraphics[width=9cm,height=5cm]{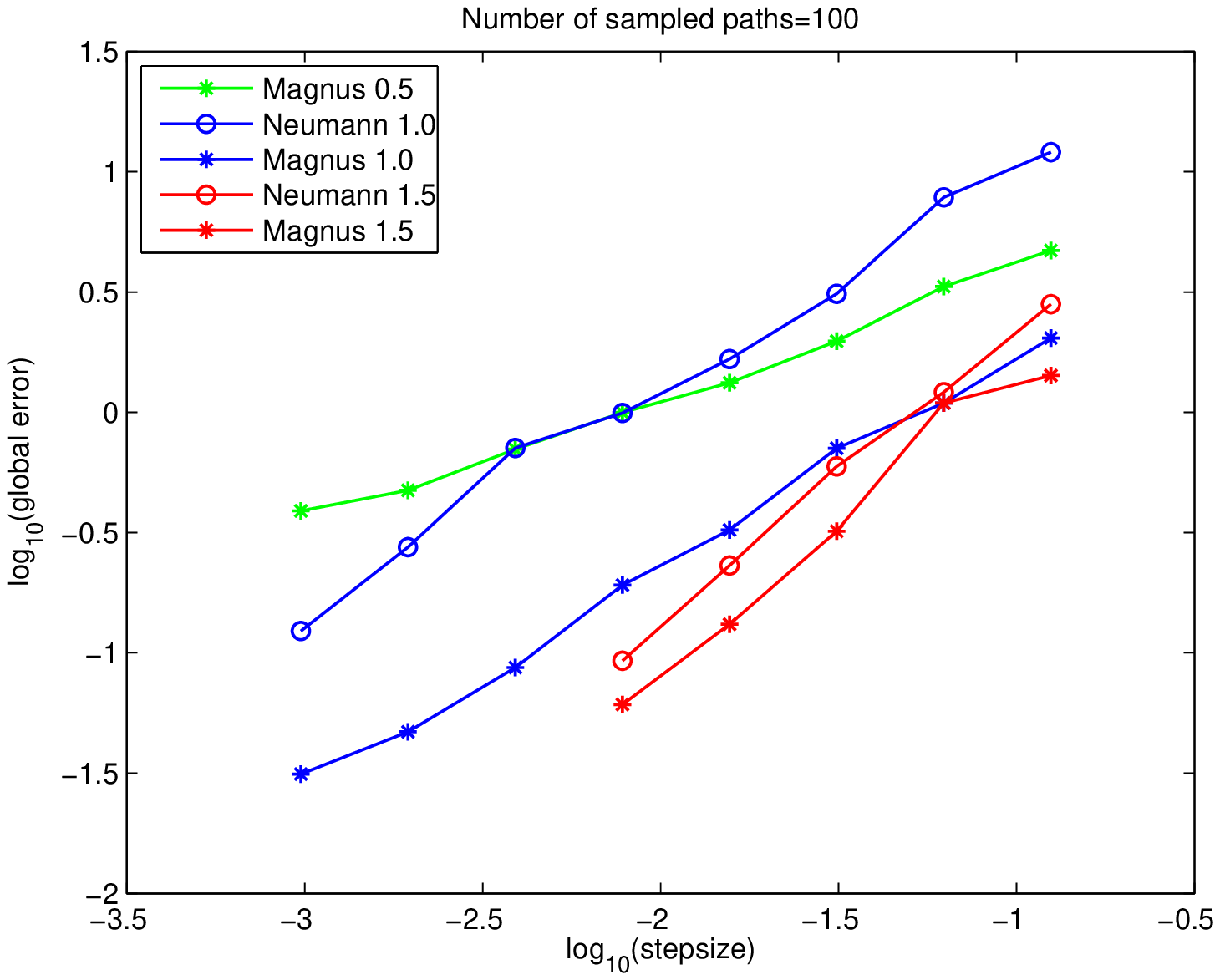}
  \includegraphics[width=9cm,height=5cm]{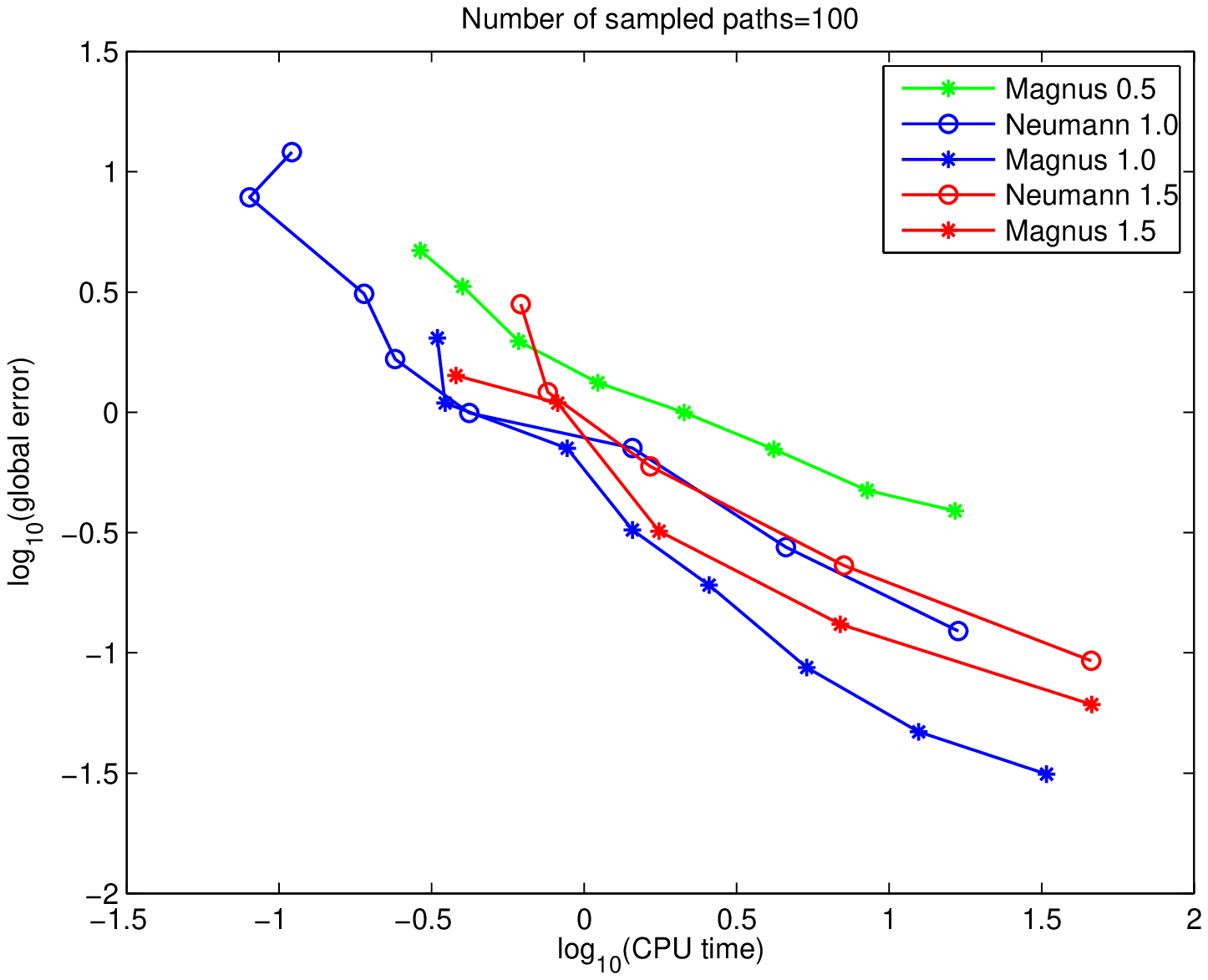}
  \end{center}
  \caption{Global error vs stepsize (top) 
  and vs CPU clocktime (bottom) 
  for the model problem at time $t=1$ with $3$ driving Wiener
  processes. The error corresponding to the 
  largest step size takes the shortest time to compute.}
  \label{globalerror3}
\end{figure}

Our second numerical example is for a homogeneous 
and constant coefficient linear problem involving two Wiener
processes with coefficient matrices $a_i\equiv D_i$
where the $D_i$, $i=0,1,2$ are given in~\eqref{mcoeffs}
and do \emph{not commute},
and initial data $y_0=(\tfrac12 \,\, 1)^T$.
In Figure~\ref{globalerror} we show how the error scales with
stepsize and CPU clocktime. We see that the superior accuracy of the Magnus
integrators is achieved for the same computational cost.
Note that in the case $M=1$ we have $d=\dpp=2$. 
For the case when $h\leq h_{\mathrm{cr}}$, when 
computational cost is dominated by quadrature effort, 
the relation between the global error $\mE_1$ and computational
cost\/ $\mU$ is, ignoring $\mU^\eval$ and taking the logarithm,
\begin{equation*}
\log \mE_1\approx \log K_1-\tfrac12\log\mU\,.
\end{equation*}
For the order $1/2$ Magnus method the computational
cost\/ $\mU$ is given solely by the evaluation effort 
and therefore we have
\begin{equation*}
\log \mE_{1/2}=\log K_{1/2}+\tfrac12\log T(c_{1/2}p^2+c_E)-\tfrac12\log\mU\,.
\end{equation*}
Further note that $K_{1/2}$ and $K_1$ are strictly order one, 
and that $T=1$, $p=2$, $c_{1/2}=5$ and $c_1=7$ 
(counting directly from the inegration scheme). 
In addition $c_E\approx 6p^3=48$ flops using
the $(6,6)$ Pad\'e approximation with scaling.
Substituting these
values into the difference of these last two
estimates reveals that
\begin{equation*}
\log \mE_1-\log \mE_{1/2}\approx -\tfrac12\log 68\approx -0.9\,,
\end{equation*}
which is in good agreement with the difference shown 
in Figure~\ref{globalerror}. 

Also for this example, we have shown $90\%$
confidence intervals in Figure~\ref{fig:confidence} 
for the global errors of the uniformly accurate
Magnus integrators of orders~$1$ and $3/2$. 
The confidence intervals become narrower as the 
stepsize decreases and order of the method increases,
as expected---see Kloeden and Platen~\cite[pp.~312--316]{KP}.

In Figure~\ref{globalerror3} we consider a linear 
stochastic differential system driven by \emph{three}
independent scalar Wiener processes, with the same
vector fields as in the linear system with two
Wiener processes just considered but with an additional
linear diffusion vector field characterized by
the coefficient matrix
\begin{equation*}
a_3=\begin{pmatrix} \tfrac14 &\tfrac25\\ 
                    \tfrac16 &\tfrac17 
\end{pmatrix}\,,
\end{equation*}
which does not commute with $a_0$, $a_1$ or $a_2$.
Using Table~7.1 we would expect to see for small stepsizes for
the order $1$ and $3/2$ Magnus and Neumann methods,
that the global error scales with the computational effort
with exponent $-1/2$. This can be seen in
Figure~\ref{globalerror3}.

\section{Concluding remarks}
Our results can be extended to nonlinear stochastic
differential equations. Consider a 
stochastic differential system governed
by $(d+1)$ nonlinear autonomous vector fields $V_i(y)$---instead 
of the linear vector fields `$a_i(t) y$' in~\eqref{sde}---and
driven by a $d$-dimensional Wiener process $(W^1,\ldots,W^d)$.
If we take the logarithm of the stochastic Taylor
series for the flow-map we obtain the exponential
Lie series (see Chen~\cite{Ch} and Strichartz~\cite{Str})
\begin{equation*}
\sigma_t=\sum_{i=0}^d J_i(t)V_i+\sum_{j>i=0}^d
\tfrac12(J_{ij}-J_{ji})(t)[V_i,V_j]+\cdots\,.
\end{equation*}
Here $[\cdot\,,\cdot]$ is the 
Lie--Jacobi bracket on the Lie algebra of vectors
fields defined on $\mathbb R^p$. The solution
$y_t$ of the system at time $t>0$ is given by
$y_t=\exp\sigma_t\circ y_0$
(see for example Ben Arous~\cite{BA}  
or Castell and Gaines~\cite{CG,CG2}).
Across the interval $[t_n,t_{n+1}]$ let $\hat\sigma_{t_n,t_{n+1}}$
denote the exponential Lie series truncated to a given order,
with multiple integrals approximated by their 
expectations conditioned on intervening 
sample points. Then 
\begin{equation*}
\hat y_{t_{n+1}}=\exp\bigl(\hat\sigma_{t_n,t_{n+1}}\bigr)\,\circ y_{t_n}
\end{equation*} 
is an approximation to the exact solution $y_{t_{n+1}}$ at the 
end point of the interval.
The truncated and conditioned exponential Lie series
$\hat\sigma_{t_n,t_{n+1}}$ is itself an ordinary vector field.
Hence to compute $\hat y_{t_{n+1}}$ we solve the 
ordinary differential system 
\begin{equation*}
u'(\tau)=\hat\sigma_{t_n,t_{n+1}}\circ u(\tau)
\end{equation*}
with $u(0)=y_{t_n}$ across the interval $\tau\in[0,1]$
(see Castell and Gaines~\cite{CG,CG2}). If we use 
a sufficiently accurate ordinary differential integrator
commensurate with the order of the truncation of the
exponential Lie series then $u(1)\approx \hat y_{t_{n+1}}$ 
to that order. 

Hence the order~$1/2$ exponential
Lie series integrator is more accurate than
the Euler--Maruyama method. Further 
in the case of commuting diffusion vector fields,
the uniformly accurate exponential Lie series integrators 
of order $1$ and $3/2$ are more accurate than 
the stochastic Taylor approximations of the corresponding order.
This generalization is discussed in Malham and Wiese~\cite{MWLie}.
An important future investigation to justify the
viability of these schemes would be the
relation between global error and computational cost,
which must take into account the additional 
computational effort associated with
the ordinary differential solver.

An important application for our results that
we have in mind for the future are large order
problems driven by a large number of 
Wiener processes. High-dimensional problems 
occur in many financial applications, 
for example in portfolio optimization or risk management and 
in the context of option pricing when high-dimensional
models are used, for example for the pricing of 
interest rate options or rainbow options.
Large order problems also arise when numerically
solving stochastic parabolic partial
differential equation driven by a 
multiplicative noise term which is 
white noise in time and spatially smooth
(see for example Lord and Shardlow~\cite{LS}).
Here we think of projecting the system
onto a finite spatial basis set which
results in a large system of coupled ordinary
stochastic differential equations each driven
by a multiplicative noise term. The high
dimension $d$ of the driving Wiener process 
will now be an important factor in the 
computational cost for order~$M\geq1$ as for example
we will need to simulate $\tfrac12 d(d-1)$ multiple
integrals $J_{ij}$; though the results of 
Wiktorsson~\cite{W} suggest this can be improved upon.
Krylov subspace methods for computing large matrix exponentials
would be important for efficient implementation of our
methods for this case (see Moler and Van Loan~\cite{MV} 
and Sidje~\cite{Sidje}). 

Lastly, extensions of our work that we also intend to
investigate further are: (1) implementing a variable step scheme 
following Gaines and Lyons~\cite{GL97},
Lamba, Mattingly and Stuart~\cite{LMS},
Burrage and Burrage~\cite{BBvar} and 
Burrage, Burrage and Tian~\cite{BBT}---by using 
analytic expressions for the local truncation errors 
(see Aparicio, Malham and Oliver~\cite{AMO}); 
(2) pricing path-dependent options; (3) deriving  
strong symplectic numerical methods 
(see Milstein, Repin and Tretyakov~\cite{MRT});
and (4) constructing
numerical methods based on the nonlinear Magnus expansions
of Casas and Iserles~\cite{CI}.

\section*{Acknowledgments}
We thank Kevin and Pamela Burrage, Sandy Davie, 
Terry Lyons, Per-Christian Moan, Nigel Newton, 
Tony Shardlow, Josef Teichmann and Michael Tretyakov 
for stimulating discussions. The research in this 
paper was supported by the 
EPSRC first grant number GR/S22134/01.
A significant portion of this paper was
completed while SJAM was visiting the Isaac Newton
Institute in the Spring of 2007. We would also
like to thank the anonymous referees for their
very helpful comments and suggestions.

\appendix
\section{Neumann and Magnus integrators}\label{theapp}
We present Neumann and Magnus integrators up to global order $3/2$
in the case of a $d$-dimensional Wiener process $(W^1,\ldots, W^d)$, 
and with constant coefficient matrices $a_0$ and $a_i$, $i=1,\ldots,d$.
The Neumann expansion for the solution of the
stochastic differential equation~\eqref{sde}
over an interval $[t_n,t_{n+1}]$, where $t_n=nh$, is
\begin{equation*}
y^{\text{neu}}_{t_n,t_{n+1}}\approx \bigl(I+S_{1/2}+S_1+S_{3/2}\bigr)\,y_{t_n}\,,
\end{equation*}
where (the indices $i,j,k$ run over the values $1,\ldots,d$)
\begin{align*}
S_{1/2}=&\;J_0a_0+\sum_iJ_ia_i+\sum_{i}J_{ii}a_i^2\,,\\
S_1=&\;\sum_{i\neq j}J_{ij}a_ia_j\,,\\
S_{3/2}=&\;\sum_{i}(J_{i0}a_0a_i+J_{0i}a_ia_0)
+\sum_{i,j,k}J_{kji}a_ia_ja_k\\
&\;+\sum_{i}(J_{0ii}a_i^2a_0+J_{ii0}a_0a_i^2)
+\sum_{i,j}J_{iijj}a_i^2a_j^2+J_{00}a_0^2\,.
\end{align*}
The corresponding approximation using the Magnus expansion is
\begin{equation*}
y^{\text{mag}}_{t_n,t_{n+1}}\approx\exp\bigl(s_{1/2}+s_1+s_{3/2}\bigr)\,y_{t_n}\,,
\end{equation*}
where, with $[\cdot,\cdot]$ as the matrix commutator, 
\begin{align*}
s_{1/2}=&\;J_0a_0+\sum_iJ_ia_i\,,\\
s_1=&\;\sum_{i<j}\tfrac12(J_{ji}-J_{ij})[a_i,a_j]\,,\\
s_{3/2}=&\;\sum_{i}\tfrac12(J_{i0}-J_{0i})[a_0,a_i]
+\sum_{i\neq j}(J_{iij}-\tfrac{1}{2}J_iJ_{ij}
+\tfrac{1}{12}J_i^2J_j)[a_i,[a_i,a_j]]\\
&\;+\sum_{i<j<k} 
\bigl((J_{ijk}+\tfrac12 J_jJ_{ki}+\tfrac12 J_kJ_{ij}
-\tfrac23 J_iJ_jJ_k)[a_i,[a_j,a_k]]\\
&\;\qquad\qquad+(J_{jik}+\tfrac12 J_iJ_{kj}+\tfrac12 J_kJ_{ji}
-\tfrac23 J_iJ_jJ_k)[a_j,[a_i,a_k]]\bigr)\\
&\;+\sum_{i}(J_{ii0}-\tfrac{1}{2}J_iJ_{i0}
+\tfrac{1}{12}J_i^2J_0)[a_i,[a_i,a_0]]\,.
\end{align*}
To obtain a numerical scheme of global order $M$ 
using the Neumann or Magnus expansion, we must 
use all the terms up to and including $S_M$ 
or $s_M$, respectively. Leading order terms of order $M+1/2$
with non-zero expectation can be replaced by their expectations 
(as detailed at the end of Section~\ref{sec:globalerror}).
Further, extending these solution series to
the non-homogeneous and/or the non-constant coefficient
case is straightforward.


\label{lastpage}

\end{document}